\documentclass{amsart}

\usepackage{mathrsfs}
\usepackage[arrow,matrix]{xy}
\usepackage[all]{xypic}
\usepackage{amsfonts}
\usepackage{amsmath}
\usepackage{amsthm}
\usepackage{amssymb}
\usepackage[arrow,matrix]{xy} 
\usepackage[all]{xypic}
\usepackage{graphicx}

\newtheorem{theorem}{Theorem}[section]
\newtheorem{lemma}[theorem]{Lemma}
\newtheorem{proposition}[theorem]{Proposition}
\newtheorem{corollary}[theorem]{Corollary}

\theoremstyle{definition}

\newtheorem{example}[theorem]{Example}

\theoremstyle{remark}

\numberwithin{equation}{section}

\providecommand{\rad}{\mathop{\rm rad}\nolimits}
\renewcommand{\mod}{\mathop{\rm mod}\nolimits}%
\providecommand{\add}{\mathop{\rm add}\nolimits}%
\providecommand{\ann}{\mathop{\rm ann}\nolimits}%
\providecommand{\gldim}{\mathop{\rm gl.dim}\nolimits}%
\providecommand{\pd}{\mathop{\rm pd}\nolimits}%
\providecommand{\id}{\mathop{\rm id}\nolimits}%
\providecommand{\op}{\mathop{\rm op}\nolimits}%
\providecommand{\Hom}{\mathop{\rm Hom}\nolimits}%
\providecommand{\Ext}{\mathop{\rm Ext}\nolimits}%
\providecommand{\Tor}{\mathop{\rm Tor}\nolimits}%
\providecommand{\TrD}{\mathop{{\rm Tr}D}\nolimits}%
\providecommand{\DTr}{\mathop{D\rm Tr}\nolimits}%
\providecommand{\End}{\mathop{\rm End}\nolimits}%
\providecommand{\soc}{\mathop{\rm soc}\nolimits}%
\providecommand{\rep}{\mathop{\rm rep}\nolimits}%
\renewcommand{\top}{\mathop{\rm top}\nolimits}%

\title[Components without external short paths]{On Auslander-Reiten components of algebras without external short paths}

\author{Alicja Jaworska}
\address{Faculty of Mathematics and Computer Science, Nicolaus Copernicus University, Chopina 12/18, 87-100 Toru\'n, Poland}
\email{jaworska@mat.uni.torun.pl}
\thanks{The research supported by the Research Grant N N201 269135 of the Polish Ministry of Science and Higher Education.}
\author{Piotr Malicki}
\address{Faculty of Mathematics and Computer Science, Nicolaus Copernicus University, Chopina 12/18, 87-100 Toru\'n, Poland}
\email{pmalicki@mat.uni.torun.pl}

\author{Andrzej Skowro\'nski}
\address{Faculty of Mathematics and Computer Science, Nicolaus Copernicus University, Chopina 12/18, 87-100 Toru\'n, Poland}
\email{skowron@mat.uni.torun.pl}

\dedicatory{Dedicated to Daniel Simson on the occasion of his seventieth birthday.}

\begin{document}
\maketitle

\begin{abstract}
We describe the structure of semi-regular Auslander-Reiten components of artin algebras without
external short paths in the module category. As an application we give a complete description of
self-injective artin algebras whose Auslander-Reiten quiver admits a regular acyclic component
without external short paths.
\end{abstract}

\section{Introduction and the main results}
\label{intro}

\noindent Throughout the paper, by an algebra we mean a basic connected artin algebra over a commutative
artin ring $K$. For an algebra $A$, we denote by $\mod A$ the category of finitely generated
right $A$-modules  and by $\rad_A$ the radical of $\mod A$, generated by all non-isomorphisms
between indecomposable modules in $\mod A$. Then the infinite radical $\rad_{A}^{\infty}$ of
$\mod A$ is the intersection of all powers $\rad^i_A$, $i \geqslant 1$, of $\rad_A$. By a result
of Auslander \cite{A}, $\rad_A^{\infty}=0$ if and only if $A$ is of finite representation type,
that is, there are in $\mod A$ only finitely many indecomposable modules up to isomorphism.
Moreover, we denote by $\Gamma_A$ the Auslander-Reiten quiver of $A$ and by $\tau_A$ and $\tau^-_A$
the Auslander-Reiten translations $\DTr$ and $\TrD$  in $\mod A$, respectively. We do not
distinguish between an indecomposable module $X$ in $\mod A$ and the corresponding vertex $\{X\}$
in $\Gamma_A$. Moreover, by a component of $\Gamma_A$ we mean a connected component of the quiver
$\Gamma_A$.

The Auslander-Reiten quiver $\Gamma_A$ of an algebra $A$ is an important combinatorial and homological
invariant of its module category $\mod A$. Frequently, algebras can be recovered from the graph
structure, for example the shape of components, of their Auslander-Reiten quivers. Further, very often
the behaviour of components of the Auslander-Reiten quiver $\Gamma_A$ of an algebra $A$ in the
category $\mod A$ leads to essential homological information on $A$, allowing to determinate $A$ and
$\mod A$ completely. Recall that a component $\mathcal{C}$ of an Auslander-Reiten quiver $\Gamma_A$
is called \emph{regular} if $\mathcal{C}$ contains neither a projective module nor an injective
module, and \emph{semi-regular} if $\mathcal{C}$ does not contain both a projective module and an
injective module. By general theory (see \cite{L1}, \cite{L2}, \cite{Z}), every regular component
$\mathcal{C}$ of $\Gamma_A$ is either of the form $\mathbb{Z}\Delta$, for a locally finite acyclic
valued quiver $\Delta$, or a \emph{stable tube} $\mathbb{Z}\mathbb{A}_{\infty}/ (\tau^r)$, for some
$r \geqslant 1$. More generally (see \cite{L2}), every semi-regular component $\mathcal{C}$ of
$\Gamma_A$ is either a full translation subquiver of such a translation quiver $\mathbb{Z}\Delta$
(acyclic case) or is a \emph{ray tube}, obtained from a stable tube by a finite number (possibly
empty) of ray insertions, or a \emph{coray tube}, obtained from a stable tube by a finite number
(possibly empty) of coray insertions.

A prominent role in the representation theory of algebras is played by the generalized standard
Auslander-Reiten components. Following \cite{S2} a component $\mathcal{C}$ of an Auslander-Reiten
quiver $\Gamma_A$ is called \emph{generalized standard} if $\rad_{A}^{\infty}(X,Y)=0$ for all modules
$X$ and $Y$ in $\mathcal{C}$. It has been proved in \cite{S2} that every generalized standard
component  $\mathcal{C}$ of $\Gamma_A$ is \emph{almost periodic}, that is, all but finitely many
$\tau_A$-orbits in $\mathcal{C}$ are periodic. Distinguished classes of generalized standard
components are formed by the Auslander-Reiten quivers of all algebras of finite representation type,
the connecting components of tilted algebras \cite{HRi} (respectively, double tilted algebras \cite{RS2},
generalized double tilted algebras \cite{RS3}), the separating families of tubes of quasi-tilted algebras
of canonical type \cite{LP2}, \cite{LS}, or more generally, separating families of almost cyclic coherent
components of generalized multicoil algebras \cite{MS}. The acyclic generalized standard components
have been described completely in \cite{S1}. In particular, the regular acyclic generalized standard
components are exactly the connecting components of tilted algebras given by regular tilting modules
\cite{S2}. On the other hand, the description of the support algebras of arbitrary generalized standard
components is an exciting but difficult problem (see \cite{S7}, \cite{S9}). Namely, it is shown in
\cite{S7} and \cite{S9} that every algebra $\Lambda$ over a field $K$ is a factor algebra of an algebra
$A$ (even symmetric algebra) with $\Gamma_A$ having a sincere generalized standard stable tube. Another
interesting open problem is to find handy criteria for an almost periodic Auslander-Reiten component
to be generalized standard. For stable tubes (respectively, almost acyclic components) such handy
criteria have been established in \cite{R2}, \cite{S2}, \cite{S7} (respectively, \cite{L4}, \cite{RS2},
\cite{RS3}, \cite{S1}).

In this paper we are concerned with the structure of components of the Auslander-Reiten quiver $\Gamma_A$
of an algebra $A$ having an ordered interaction with other components of $\Gamma_A$. Following \cite{RS1}, by an
\emph{external short path} of  a component $\mathcal{C}$ of $\Gamma_A$ we mean a sequence
$X \rightarrow Y \rightarrow Z$ of non-zero non-isomorphisms between indecomposable modules in $\mod A$
with $X$ and $Z$ in $\mathcal{C}$ but $Y$ not in $\mathcal{C}$.
We mention that every component $\mathcal{C}$ of $\Gamma_A$ without external short paths has the important
property: the additive category  $\add(\mathcal{C})$ of $\mathcal{C}$ is closed under extensions
in $\mod A$. For a component  $\mathcal{C}$
of $\Gamma_A$, we denote by $\ann_A(\mathcal{C})$ the \emph{annihilator} of $\mathcal{C}$ in $A$,
that is, the intersection of the annihilators $\ann_A(X)=\{a \in A|Xa=0\}$ of all modules $X$ in
$\mathcal{C}$. A component $\mathcal{C}$ of $\Gamma_A$ with $\ann_A(\mathcal{C})=0$ is called
{\em faithful}. We note that $\ann_A(\mathcal{C})$ is an ideal of $A$, $\mathcal{C}$ is a faithful
component of $\Gamma_{A/ \ann_A(\mathcal{C})}$, and every simple right $(A/\ann_A(\mathcal{C}))$-module
occurs as a composition factor of a module in $\mathcal{C}$.

We are now in a position to formulate the main results of the paper.

\begin{theorem} \label{thm1}
Let $A$ be an algebra, $\mathcal{C}$ a component  of $\Gamma_A$ without projective modules and
external short paths, and $B=A/\ann_A(\mathcal{C})$. Then one of the following statements holds:
\begin{itemize}
\item[(i)] $B$ is a tilted algebra of the form $\End_H(T)$, where $H$ is a hereditary algebra
and $T$ is a tilting $H$-module without non-zero preinjective direct summands, and $\mathcal{C}$
is the connecting component $\mathcal{C}_T$ of $\Gamma_B$ determined by $T$.
\item[(ii)] $B$ is the opposite algebra of an almost concealed canonical algebra and
$\mathcal{C}$ is a faithful coray tube of a separating family of coray tubes of $\Gamma_B$.
\end{itemize}
\end{theorem}

\begin{theorem} \label{thm2}
Let $A$ be an algebra, $\mathcal{C}$ a component of $\Gamma_A$ without injective modules and external
short paths, and $B=A/ \ann_A(\mathcal{C})$. Then one of the following statements holds:
\begin{itemize}
\item[(i)] $B$ is a tilted algebra of the form $\End_H(T)$, where $H$ is a hereditary algebra
and $T$ is a tilting $H$-module without non-zero preprojective direct summands, and $\mathcal{C}$
is the connecting component $\mathcal{C}_T$ of $\Gamma_B$ determined by $T$.
\item[(ii)] $B$ is an almost concealed canonical algebra and
$\mathcal{C}$ is a faithful ray tube of a separating family of ray tubes of $\Gamma_B$.
\end{itemize}
\end{theorem}

\begin{corollary} \label{cor3}
Let $A$  be an algebra, $\mathcal{C}$ a regular component of $\Gamma_A$ without
external short paths, and $B=A/\ann_A(\mathcal{C})$. Then one of the following statements holds:
\begin{itemize}
\item[(i)] $B$ is a tilted algebra of the form $\End_H(T)$, where $H$ is a hereditary algebra and
$T$ is a regular tilting
$H$-module, and $\mathcal{C}$ is the connecting component $\mathcal{C}_T$ of $\Gamma_B$ determined
by $T$.
\item[(ii)] $B$ is a concealed canonical algebra and $\mathcal{C}$ is a stable tube
of a separating family of stable tubes of $\Gamma_B$.
\end{itemize}
\end{corollary}

We would like to mention that, by a result of Ringel \cite{R3a}, a hereditary algebra $H$ admits
a regular tilting module $T$ if and only if $H$ is neither of Dynkin type nor Euclidean type and has at least
three pairwise non-isomorphic simple modules (see also \cite{Ba1}, \cite{Ba2}, \cite[Section XVIII.5]{SS2}
for constructions of regular tilting modules over wild hereditary algebras). We refer to \cite{KS}
for constructions of regular connecting components of tilted algebras and stable tubes of concealed canonical
algebras having all indecomposable modules with every simple module occuring arbitrary many times as a composition factor.
Moreover, we refer to \cite{S5} for results on the composition factors of modules in generalized
standard stable tubes.

As an immediate consequence of Theorems \ref{thm1} and \ref{thm2} we obtain the following fact.

\begin{corollary} \label{cor4}
Let $A$ be an algebra and $\mathcal{C}$ be a semi-regular component of $\Gamma_A$ without external short paths.
Then $\mathcal{C}$ is a generalized standard component of $\Gamma_A$.
\end{corollary}

We exhibit in Section 2 (Example \ref{ex2.6}) an Auslander-Reiten component without external short paths which
is not generalized standard.
It would be interesting to know when an Auslander-Reiten component without
external short paths is generalized standard. This question can be interpreted in the following way. Following \cite{S3},
a \emph{component quiver} $\Sigma_A$ of an algebra $A$ has the components of $\Gamma_A$ as the vertices and two
components $\mathcal{C}$ and $\mathcal{D}$ of $\Gamma_A$ are linked in $\Sigma_A$ by an arrow $\mathcal{C}
\rightarrow \mathcal{D}$ if $\rad_A^{\infty}(X,Y)\neq 0$ for some modules $X$ in $\mathcal{C}$ and $Y$ in
$\mathcal{D}$. In particular, a component $\mathcal{C}$ of $\Gamma_A$ is generalized standard if and only if
$\Sigma_A$ has no loop at $\mathcal{C}$. Therefore, we ask when, for a component $\mathcal{C}$ of $\Gamma_A$,
the absence of short cycles
$\xymatrix{ \mathcal{C} \ar@<0.5ex>[r] & \mathcal{D} \ar@<0.5ex>[l]}$ in $\Sigma_A$ with $\mathcal{C}\neq
\mathcal{D}$ forces the absence of loop $\xymatrix {\mathcal{C} \ar@(ru,rd)} \qquad$ at $\mathcal{C}$ in $\Sigma_A$.

It has been proved in \cite[Theorem 3.6]{S2} that the Auslander-Reiten quiver $\Gamma_A$ of an algebra $A$
has at most finitely many acyclic regular generalized standard components. Hence, we obtain the following
immediate consequence of the above results.

\begin{corollary} \label{cor5}
Let $A$ be an algebra. Then all but finitely many components of $\Gamma_A$ without external short paths
are stable tubes.
\end{corollary}

A crucial role in our proofs of Theorems \ref{thm1} and \ref{thm2} is played by the following theorem describing the components
without external short paths in the Auslander-Reiten quivers of quasi-tilted algebras. Recall that the quasi-tilted
algebras are those of the form $\End_{\mathcal{H}}(T)$ for tilting objects $T$ in hereditary abelian $\Ext$-finite
categories $\mathcal{H}$, or equivalently, the algebras $\Lambda$ of global dimension at most two and with
every indecomposable module in $\mod \Lambda$ of the projective dimension or the injective dimension at most  one
\cite{HRS}.

\begin{theorem} \label{thm6}
Let $A$ be a quasi-tilted algebra and $\mathcal{C}$  a component of $\Gamma_A$. The following statements are equivalent:
\begin{itemize}
\item[(i)] $\mathcal{C}$ has no external short path.
\item[(ii)] $\mathcal{C}$ is almost periodic.
\item[(iii)] $\mathcal{C}$ is generalized standard.
\item[(iv)] $\mathcal{C}$ is either a preprojective component, a preinjective component, a ray tube, a coray tube,
or a connecting component (in case $A$ is a tilted algebra).
\end{itemize}
\end{theorem}

We would like to mention that, in general, there are many
generalized standard (even faithful) components having external
short paths, for example the faithful generalized standard stable
tubes over generalized canonical (but not canonical) algebras
introduced in \cite{S7} (see also \cite{S9}). On the other hand,
Theorem \ref{thm6} leads to a similar characterization of components
without external short  paths in the Auslander-Reiten quivers of
generalized double tilted algebras investigated in \cite{RS2},
\cite{RS3}. In particular, it is the case for all algebras $A$
with all but finitely many indecomposable modules in $\mod A$ of
the projective dimension (respectively, injective dimension) at
most one \cite{S4}. We refer also to \cite{S8} for a
characterization of the class of algebras consisting of the
quasi-tilted algebras and generalized double tilted algebras.
Finally, we also mention that Theorem \ref{thm6} leads to a similar
characterization of components without external short paths in the
Auslander-Reiten quivers of algebras having separating families of
almost cyclic coherent components, where the connecting components
have to be replaced by the generalized multicoils (see \cite{MS}
for details). On the other hand, it is not clear if every
Auslander-Reiten component without external short paths
is almost periodic.\\

From Drozd's Tame and Wild Theorem \cite{Dr} the class of finite dimensional algebras over an algebraically closed
field $K$ may be divided into two disjoint classes. One class consists of the tame algebras for which
the indecomposable modules occur, in each dimension $d$, in a finite number of discrete and a finite number of
one-parameter families. The second class is formed by the wild algebras whose representation theory `comprises'
the representation theories of all finite dimensional algebras over $K$. Among the tame algebras we may  distinguish
the  classes of algebras of polynomial growth \cite{S0} (respectively, domestic \cite{CB2}, \cite{S0}) for which there
exists a positive integer $m$ such that the indecomposable modules occur, in each dimension $d$, in a finite
number of discrete and at most $d^m$ (respectively, at most $m$) one-parameter families. Moreover, it has been proved
by Crawley-Boevey \cite{CB1} that, for a tame algebra $A$, all but finitely many indecomposable $A$-modules
of any fixed dimension $d$ lie in stable tubes of rank one. We refer to \cite[Chapter XIX]{SS2} for precise definitions
and properties of tame and wild algebras.

The following result is an immediate consequence of Corollary \ref{cor3}, Theorem \ref{thm6}, the above theorem of Crawley-Boevey,
and the fact that the tilted algebras given by regular tilting modules and the concealed canonical algebras
of wild type are wild.
\begin{corollary} \label{cor7}
Let $A$ be a finite dimensional tame algebra over  an algebraically  closed field $K$ and $\mathcal{C}$
a regular  component of $\Gamma_A$ without external short paths. Then $\mathcal{C}$ is a stable tube
and $A/ \ann_A(\mathcal{C})$ is a tame concealed algebra or a tubular algebra.
\end{corollary}

Moreover, we have also the following consequence of Corollary \ref{cor3}, Theorem \ref{thm6}, \cite{CB1}, \cite{R2},
\cite[Lemma 3.6]{S0} and \cite[Theorem A and Corollary B]{S6}.

\begin{corollary} \label{cor8}
Let $A$ be a finite dimensional tame algebra over  an algebraically closed field $K$  such that no
component of $\Gamma_A$ has an external short path. Then the following facts hold.
\begin{itemize}
\item[(i)] $A$ is of polynomial growth.
\item[(ii)] $A$ is a domestic algebra if and only if all but finitely  many components of $\Gamma_A$
are stable tubes of rank one.
\end{itemize}
\end{corollary}

It would be important to know if a finite dimensional algebra $A$ over  an algebraically closed field
$K$ with every component in $\Gamma_A$ without external short paths is actually a tame algebra. On the
other hand, if for such an algebra $A$ every component in $\Gamma_A$ is generalized standard, then
$\rad_A^{\infty}(X,X)=0$ for any indecomposable module $X$ in $\mod A$, and consequently $A$ is a tame
algebra (see \cite[Proposition 3.3]{S?}). Hence we obtain, by Corollary \ref{cor8}, the following fact.

\begin{corollary}\label{cor9}
Let $A$ be a finite dimensional algebra over  an algebraically closed field $K$  such that
the component  quiver  $\Sigma_A$ of $A$ has no short cycles
$\xymatrix{\circ  \ar@<0.5ex>[r] & \ar@<0.5ex>[l] \circ}$ and no loops  $\xymatrix {\circ \ar@(ru,rd)}\qquad$. Then
$A$ is a tame algebra of polynomial growth.
\end{corollary}

As an application of Corollary \ref{cor3} and results established in \cite{EKS}, \cite{RSS2} and \cite{SY1}, \cite{SY2}, \cite{SY3},
we obtain a complete description of all self-injective algebras whose Auslander-Reiten quiver admits a regular
acyclic component without external short paths.

\begin{theorem}\label{thm10}
Let $A$ be a self-injective algebra. The following statements are equivalent.
\begin{itemize}
\item[(i)] $\Gamma_A$ admits a regular acyclic component $\mathcal{C}$ without external short paths.
\item[(ii)] $A$ is isomorphic to an orbit algebra $\widehat{B}/ (\varphi\nu^2_{\widehat{B}})$,
where $\widehat{B}$ is the repetitive category of a tilted algebra  $B$ of the form $\End_H(T)$, for some
hereditary algebra $H$ and a regular tilting $H$-module $T$, $\nu_{\widehat{B}}$ is the Nakayama automorphism
of $\widehat{B}$, and $\varphi$ is a positive automorphism of $\widehat{B}$.
\end{itemize}
\end{theorem}

We refer to \cite{EKS} for the representation theory of orbit algebras $\widehat{B}/ G$ of the repetitive
categories $\widehat{B}$ of tilted algebras $B$ of wild type and infinite cyclic automorphism groups $G$
of $\widehat{B}$. We also note that the problem of describing the self-injective algebras whose
Auslander-Reiten quiver admits a stable tube without external short paths is more difficult, because the
stable tubes occur in families of quasi-tubes. We refer to \cite{KSY} for a wide class of self-injective algebras
having infinitely many stable tubes without external short paths.

The paper is organized as follows. In Section 2 we prove Theorem \ref{thm6}
and recall the related background on tilted algebras and
quasi-tilted algebras of canonical type. In Section 3 we provide
the proofs of Theorems \ref{thm1} and \ref{thm2}, showing that every algebra whose
Auslander-Reiten quiver admits a faithful semi-regular component
without external short paths is a quasi-tilted algebra, and then
applying Theorem \ref{thm6}. The final Section 4 is devoted to the proof of
Theorem \ref{thm10} and the related background on the orbit algebras of
repetitive categories of algebras.

For background on the representation theory we refer to the books \cite{ASS}, \cite{ARS}, \cite{H}, \cite{R2},
\cite{SS1} and \cite{SS2}.

\section{Proof of Theorem \ref{thm6}}

In the proof of Theorem \ref{thm6} we need translation subquivers of the Auslander-Reiten quivers of special type.
Let $A$ be an algebra, $\mathcal{C}$ a component of $\Gamma_A$ and $M$ an indecomposable module in
$\mathcal{C}$. Then the {\em left cone} $(\rightarrow M)$ of $M$ is the full translation subquiver
of $\mathcal{C}$ formed by all predecessors of $M$ in $\mathcal{C}$ and the {\em right cone}
$(M \rightarrow)$ of $M$ is the full translation subquiver  of $\mathcal{C}$ formed by all successors
of $M$ in $\mathcal{C}$.

It has been proved in \cite{HRe} that the class of quasi-tilted algebras consists of the tilted algebras
\cite{HRi} (endomorphism algebras of tilting modules over hereditary algebras) and the quasi-tilted algebras
of canonical type \cite{LS} (endomorphism algebras of tilting objects in hereditary abelian categories
whose bounded derived category is equivalent to the bounded derived category of a canonical algebra
in the sense of Ringel \cite{R2}, \cite{R4}). Accordingly, we will divide the proof of Theorem \ref{thm6} into two cases:
the tilted case and the canonical case.

Let $H$ be a hereditary algebra, $Q_H$ the valued quiver of $H$, and $T$ a multiplicity-free tilting module in $\mod H$,
that is, $\Ext^1_H(T,T)=0$ and $T$ is a direct sum of $n$ pairwise non-isomorphic indecomposable $H$-modules
with $n$ the rank of the Grothendieck group $K_0(H)$ of $H$. Consider the associated tilted algebra $B=\End_H(T)$
of type $Q_H$. Then the tilting module $T$ determines the torsion pair $(\mathcal{F}(T), \mathcal{T}(T))$ in
$\mod H$, with the torsion-free part $\mathcal{F}(T)=\{X \in \mod H | \Hom_H(T,X)=0\}$ and the torsion part
$\mathcal{T}(T)=\{X \in \mod H | \Ext^1_H(T,X)=0\}$, and the splitting torsion pair $(\mathcal{Y}(T), \mathcal{X}(T))$
in $\mod B$, with the torsion-free part $\mathcal{Y}(T)=\{Y \in \mod B| \Tor^B_1(Y,T)=0\}$ and the torsion part
$\mathcal{X}(T)=\{Y \in \mod B| Y \otimes_B T=0\}$. Moreover, by the Brenner-Butler theorem, the functor
$\Hom_H(T,-): \mod H \rightarrow \mod B$ induces an equivalence of $\mathcal{T}(T)$ with $\mathcal{Y}(T)$,
and the functor $\Ext^1_H(T,-): \mod H \rightarrow \mod B$ induces an equivalence of $\mathcal{F}(T)$
with $\mathcal{X}(T)$ (see \cite{ASS}, \cite{HRi}). Further, the images $\Hom_H(T,I)$ of the indecomposable
injective modules $I$ in $\mod H$ via the functor $\Hom_H(T,-)$ belong to one component $\mathcal{C}_T$ of
$\Gamma_B$, called the {\em connecting component} of $\Gamma_B$ determined by $T$, and form a faithful
section $\Delta_T \cong Q^{\op}_H$ of $\mathcal{C}_T$. The section $\Delta_T$ of $\mathcal{C}_T$ has the
distinguished property: it connects the torsion-free part $\mathcal{Y}(T)$ with the torsion part $\mathcal{X}(T)$,
because every indecomposable predecessor of a module $\Hom_H(T,I)$ from $\Delta_T$ in $\mod B$ lies in
$\mathcal{Y}(T)$ and every indecomposable successor of a module $\tau^-_B \Hom_H(T,I)$ in $\mod B$ lies in
$\mathcal{X}(T)$. We have also the following properties of the connecting component $\mathcal{C}_T$
established in \cite{R3}:
\begin{itemize}
\item[$\bullet$] $\mathcal{C}_T$ contains a projective module if and only if $T$ admits a preinjective
indecomposable direct summand;
\item[$\bullet$] $\mathcal{C}_T$ contains an injective  module if and only if $T$ admits a preprojective
indecomposable direct summand;
\item[$\bullet$] $\mathcal{C}_T$ is regular if and only if $T$ is regular.
\end{itemize}
We also mention that the Auslander-Reiten quiver $\Gamma_H$ of $H$ has the decomposition
\[\Gamma_H = \mathcal{P}(H) \vee \mathcal{R}(H) \vee \mathcal{Q}(H),\]
where $\mathcal{P}(H)$ is the preprojective component containing all indecomposable projective $H$-modules,
$\mathcal{Q}(H)$ is the preinjective component containing all indecomposable injective $H$-modules,
and $\mathcal{R}(H)$ is the family of all regular components. Moreover, we have
\begin{itemize}
\item[$\bullet$] If $Q_H$ is a Dynkin quiver, then $\mathcal{R}(H)$ is empty and $\mathcal{P}(H)=
\mathcal{Q}(H)$.
\item[$\bullet$] If $Q_H$ is a Euclidean quiver, then $\mathcal{P}(H)\cong (-\mathbb{N})Q^{\op}_H$,
 $\mathcal{Q}(H)\cong \mathbb{N}Q^{\op}_H$ and $\mathcal{R}(H)$ is an infinite family of pairwise
 orthogonal generalized standard stable tubes \cite{R2}, \cite{SS1}.
\item[$\bullet$] If $Q_H$ is a wild quiver, then $\mathcal{P}(H) \cong (-\mathbb{N})Q^{\op}_H$,
$\mathcal{Q}(H)\cong \mathbb{N}Q^{\op}_H$ and $\mathcal{R}(H)$ is an infinite family of regular
components of type $\mathbb{ZA}_{\infty}$ \cite{ABPRS}, \cite{R1}, \cite{SS2}.
\end{itemize}

Let $H$ be a hereditary algebra not of Dynkin type and $T$ a multiplicity-free tilting $H$-module from
the additive category $\add(\mathcal{P}(H))$ of the preprojective component $\mathcal{P}(H)$ of
$\Gamma_H$. Then $B=\End_H(T)$ is called a {\em concealed algebra} of type $Q_H$. A concealed algebra
$B=\End_H(T)$ is called a {\em tame concealed } algebra if $Q_H$ is a Euclidean quiver, and a {\em wild concealed
algebra} if $Q_H$ is a wild quiver.\\

The following fact proved by Baer \cite{Ba1} (see also \cite[Theorem XVIII.2.6]{SS2}) will be important
for our considerations.
\begin{lemma}
Let $H$ be a wild hereditary algebra, and $X$, $Y$ be two indecomposable modules in $\mathcal{R}(H)$.
Then there is a positive integer $m$ such that $\Hom_H(X, \tau_H^rY)\neq 0$ for all integers $r \geqslant m$.
\end{lemma}
We will prove now Theorem \ref{thm6} in the tilted case.
\begin{proposition}
Let $H$ be a hereditary algebra, $T$ a multiplicity-free tilting $H$-module, $B=\End_H(T)$ the associated
tilted algebra, and $\mathcal{C}$ a component of $\Gamma_B$. The following statements are equivalent:
\begin{itemize}
\item[(i)] $\mathcal{C}$ has no external short path.
\item[(ii)] $\mathcal{C}$ is almost periodic.
\item[(iii)] $\mathcal{C}$ is generalized standard.
\item[(iv)]  $\mathcal{C}$ is either a preprojective component, a preinjective component, a ray tube,
a coray tube, or the connecting component $\mathcal{C}_T$.
\end{itemize}
\end{proposition}
\begin{proof}
We start with the general view on the module category $\mod B$ due to results established in \cite{K1},
\cite{K2}, \cite{K3}, \cite{L3}, \cite{St}.
Let $\Delta= \Delta_T$ be the canonical section of the connecting component $\mathcal{C}_T$ determined
by $T$. Hence, $\Delta= Q^{\op}$ for $Q=Q_H$. Then $\mathcal{C}_T$ admits a finite (possibly empty)
family of pairwise disjoint full translation (valued) subquivers
\[\mathcal{D}^{(l)}_1, ..., \mathcal{D}^{(l)}_m, \mathcal{D}^{(r)}_1, ..., \mathcal{D}^{(r)}_n\]
such that  the following statements hold:
\begin{itemize}
\item[(a)] For each $i \in \{1,...,m\}$, there is an isomorphism of translation quivers $\mathcal{D}^{(l)}_i
\cong \mathbb{N} \Delta^{(l)}_i$, where $\Delta^{(l)}_i$ is a connected full valued subquiver of $\Delta$,
and $\mathcal{D}^{(l)}_i$ is closed under predecessors in $\mathcal{C}_T$.
\item[(b)] For each $j \in \{1,...,n\}$, there is an isomorphism of translation quivers $\mathcal{D}^{(r)}_j
\cong (-\mathbb{N}) \Delta^{(r)}_j$, where $\Delta^{(r)}_j$ is a connected full valued subquiver of $\Delta$,
and $\mathcal{D}^{(r)}_j$ is closed under successors in $\mathcal{C}_T$.
\item[(c)] All but finitely many indecomposable modules of $\mathcal{C}_T$ lie in
\[\mathcal{D}^{(l)}_1 \cup ...\cup \mathcal{D}^{(l)}_m \cup \mathcal{D}^{(r)}_1 \cup ...\cup \mathcal{D}^{(r)}_n.\]
\item[(d)] For each $i \in \{1,...,m\}$, there exists a tilted algebra $B^{(l)}_i=\End_{H^{(l)}_i}(T^{(l)}_i)$,
where $H^{(l)}_i$ is a hereditary algebra of type $(\Delta^{(l)}_i)^{\op}$ and $T^{(l)}_i$ is a multiplicity-free
tilting $H^{(l)}_i$-module without preinjective indecomposable direct summands such that
\begin{itemize}
\item[$\bullet$] $B^{(l)}_i$ is a quotient algebra of $B$, and hence there is a fully faithful embedding
$\mod B^{(l)}_i \hookrightarrow \mod B$,
\item[$\bullet$] $\mathcal{D}^{(l)}_i$ coincides with the torsion-free part $\mathcal{Y}(T^{(l)}_i) \cap
\mathcal{C}_{T^{(l)}_i}$ of the connecting component $\mathcal{C}_{T^{(l)}_i}$ of $\Gamma_{B^{(l)}_i}$
determined by $T^{(l)}_i$.
\end{itemize}
\item[(e)]For each $j \in \{1,...,n\}$, there exists a tilted algebra $B^{(r)}_j=\End_{H^{(r)}_j}(T^{(r)}_j)$,
where $H^{(r)}_j$ is a hereditary algebra of type $(\Delta^{(r)}_j)^{\op}$ and $T^{(r)}_j$ is a multiplicity-free
tilting $H^{(r)}_j$-module without preprojective indecomposable direct summands such that
\begin{itemize}
\item[$\bullet$] $B^{(r)}_j$ is a quotient algebra of $B$, and hence there is a fully faithful embedding
$\mod B^{(r)}_j \hookrightarrow \mod B$,
\item[$\bullet$] $\mathcal{D}^{(r)}_j$ coincides with the torsion part $\mathcal{X}(T^{(r)}_j) \cap
\mathcal{C}_{T^{(r)}_j}$ of the connecting component $\mathcal{C}_{T^{(r)}_j}$ of $\Gamma_{B^{(r)}_j}$
determined by $T^{(r)}_j$.
\end{itemize}
\item[(f)] $\mathcal{Y}(T)=\add (\mathcal{Y}(T^{(l)}_1) \cup ... \cup \mathcal{Y}(T^{(l)}_m)\cup (\mathcal{Y}(T)
\cap \mathcal{C}_T))$.
\item[(g)] $\mathcal{X}(T)=\add ((\mathcal{X}(T) \cap \mathcal{C}_T)\cup \mathcal{X}(T^{(r)}_1) \cup ... \cup
\mathcal{X}(T^{(r)}_n))$.
\item[(h)] The Auslander-Reiten quiver $\Gamma_B$ has the disjoint union form
\[\Gamma_B = (\bigcup_{i=1}^m \mathcal{Y}\Gamma_{B^{(l)}_i}) \cup \mathcal{C}_T \cup  (\bigcup_{j=1}^n
\mathcal{X}\Gamma_{B^{(r)}_j}),\]
where
\begin{itemize}
\item[$\bullet$] for each $i \in \{1,...,m\}$, $\mathcal{Y}\Gamma_{B^{(l)}_i}$ is the union of all components
of $\Gamma_{B^{(l)}_i}$ contained entirely in $\mathcal{Y}(T^{(l)}_i)$,
\item[$\bullet$] for each $j \in \{1,...,n\}$, $\mathcal{X}\Gamma_{B^{(r)}_j}$ is the union of all components
of $\Gamma_{B^{(r)}_j}$ contained entirely in $\mathcal{X}(T^{(r)}_j)$.
\end{itemize}
\end{itemize}
Moreover, we have the following description of the components of $\Gamma_B$ contained in the parts
$\mathcal{Y}\Gamma_{B^{(l)}_i}$ and $\mathcal{X}\Gamma_{B^{(r)}_j}$:
\begin{itemize}
\item[(1)] If $\Delta^{(l)}_i$ is a Euclidean quiver, then $\mathcal{Y}\Gamma_{B^{(l)}_i}$ consists
of a unique preprojective component $\mathcal{P}({B^{(l)}_i})$ of $\Gamma_{{B^{(l)}_i}}$ and an infinite
family $\mathcal{T}^{B^{(l)}_i}$ of pairwise orthogonal generalized standard ray tubes. Further,
$\mathcal{P}({B^{(l)}_i})$ coincides with the preprojective component $\mathcal{P}({C^{(l)}_i})$
of a tame concealed quotient algebra $C^{(l)}_i$ of $B^{(l)}_i$.
\item[(2)] If $\Delta^{(l)}_i$ is a wild quiver, then $\mathcal{Y}\Gamma_{B^{(l)}_i}$ consists of
a unique preprojective component $\mathcal{P}(B^{(l)}_i)$ of $\Gamma_{B^{(l)}_i}$ and an infinite
family of components obtained from the components of the form $\mathbb{ZA}_{\infty}$ by a finite
number (possibly empty) of ray insertions. Further, $\mathcal{P}(B^{(l)}_i)$ coincides with the
preprojective component $\mathcal{P}(C^{(l)}_i)$ of a wild concealed quotient algebra $C^{(l)}_i$
of $B^{(l)}_i$.
\item[(3)] If $\Delta^{(r)}_j$ is a Euclidean quiver, then $\mathcal{X}\Gamma_{B^{(r)}_j}$ consists
of a unique preinjective component $\mathcal{Q}(B^{(r)}_j)$ of $\Gamma_{B^{(r)}_j}$ and an infinite family
of pairwise orthogonal generalized standard coray tubes. Further, $\mathcal{Q}(B^{(r)}_j)$ coincides with
the preinjective component $\mathcal{Q}(C^{(r)}_j)$ of a tame concealed quotient algebra $C^{(r)}_j$ of
$B^{(r)}_j$.
\item[(4)] If $\Delta^{(r)}_j$ is a wild quiver, then $\mathcal{X}\Gamma_{B^{(r)}_j}$ consists of
a unique preinjective component $\mathcal{Q}(B^{(r)}_j)$ of $\Gamma_{B^{(r)}_j}$ and an infinite family of
components obtained from the components of the form $\mathbb{ZA}_{\infty}$ by a finite number (possibly
empty) of coray insertions. Further, $\mathcal{Q}(B^{(r)}_j)$ coincides with the preinjective component
$\mathcal{Q}(C^{(r)}_j)$ of a wild concealed quotient algebra $C^{(r)}_j$ of $B^{(r)}_j$.
\end{itemize}
It follows from the above facts that the preprojective components, preinjective components, ray tubes and
coray tubes of $\Gamma_B$ are generalized standard, without external short paths, and clearly are almost
periodic. On the other hand, the components of $\Gamma_B$ obtained from the components of the form
$\mathbb{ZA}_{\infty}$ by ray insertions or coray insertions are not almost periodic, and hence are not
generalized standard, by the general result \cite[Theorem 2.3]{S2}. Therefore, it remains to show that  all these
components have external short paths. We have two cases to consider.

Assume $\mathcal{C}$ is an acyclic component of $\Gamma_B$ with infinitely many $\tau_B$-orbits contained in the torsion-free
part $\mathcal{Y}(T)$ of $\mod B$. Then it follows from (1) and (2) that there is $i \in \{1,...,m\}$ such that
$\Delta^{(l)}_i$ is a wild quiver and $\mathcal{C}$ is a component of the Auslander-Reiten quiver
$\Gamma_{B^{(l)}_i}$ of the tilted algebra $B^{(l)}_i=\End_{H^{(l)}_i}(T^{(l)}_i)$ with
$H^{(l)}_i$ a wild hereditary algebra of type $(\Delta^{(l)}_i)^{\op}$ and $T^{(l)}_i$ a multiplicity-free
tilting $H^{(l)}_i$-module without preinjective indecomposable direct summands. Since $\mathcal{Y}\Gamma_{B^{(l)}_i}$
contains infinitely many components different from the preprojective component $\mathcal{P}(B^{(l)}_i)$,
we may choose a regular component $\mathcal{D}$ in $\mathcal{Y}\Gamma_{B^{(l)}_i}$ different from $\mathcal{C}$.
Clearly, $\mathcal{D}$ is of the form $\mathbb{ZA}_{\infty}$. Now it follows from \cite[Theorem 1]{K2} that  there exist
regular components $\widetilde{\mathcal{C}}$ and $\widetilde{\mathcal{D}}$ in $\Gamma_{H^{(l)}_i}$ and indecomposable
modules $X \in \mathcal{C}$, $Y \in \mathcal{D}$, $\widetilde{X} \in \widetilde{\mathcal{C}}$ and
$\widetilde{Y} \in \widetilde{\mathcal{D}}$ such that the functor $F^{(l)}_i=\Hom_{H^{(l)}_i}(T^{(l)}_i, -):
\mod H^{(l)}_i \rightarrow \mod B^{(l)}_i$ induces equivalences of the additive categories of the left
cones $\xymatrix{F^{(l)}_i:\add(\rightarrow \widetilde{X})\ar[r]^(.55){\sim}& \add(\rightarrow X)}$ and
$\xymatrix{F^{(l)}_i:\add(\rightarrow \widetilde{Y})\ar[r]^(.55){\sim}& \add(\rightarrow Y).}$
Moreover, we have $F^{(l)}_i(\tau_{H^{(l)}_i}M)=\tau_{B^{(l)}_i}F^{(l)}_i(M)$ and
$F^{(l)}_i(\tau_{H^{(l)}_i}N)=\tau_{B^{(l)}_i}F^{(l)}_i(N)$ for all modules $M$ in $(\rightarrow \widetilde{X})$
and $N$ in $(\rightarrow \widetilde{Y})$. Applying now Lemma 2.1, we obtain that there exist positive
integers $r$ and $s$ such that $\Hom_{H^{(l)}_i}(\widetilde{X}, \tau^r_{H^{(l)}_i}\widetilde{Y})\neq 0$ and
$\Hom_{H^{(l)}_i}(\tau^r_{H^{(l)}_i}\widetilde{Y}, \tau^s_{H^{(l)}_i}\widetilde{X})\neq 0$.
Hence we get $\Hom_{B^{(l)}_i}(X, \tau^r_{B^{(l)}_i}Y)\neq 0$ and $\Hom_{B^{(l)}_i}(\tau^r_{B^{(l)}_i}Y,
\tau^s_{B^{(l)}_i}X) \neq0$, and consequently an external short path $X \rightarrow \tau^r_{B^{(l)}_i}Y
\rightarrow \tau^s_{B^{(l)}_i}X$ of $\mathcal{C}$ in $\mod B^{(l)}_i$, and so in $\mod B$, because $B^{(l)}_i$
is a quotient algebra of $B$ and there is a fully faithful embedding $\mod B^{(l)}_i \hookrightarrow \mod B$.

Assume $\mathcal{C}$ is an acyclic component with infinitely many $\tau_B$-orbits contained in the torsion part
$\mathcal{X}(T)$ of $\mod B$. Then it follows from (3) and (4) that there is $j \in \{1,...,n\}$ such that $\Delta^{(r)}_j$
is a wild quiver and $\mathcal{C}$ is a component of the Auslander-Reiten quiver $\Gamma_{B^{(r)}_j}$ of the
tilted algebra $B^{(r)}_j=\End_{H^{(r)}_j}(T^{(r)}_j)$ with $H^{(r)}_j$ a wild hereditary algebra of type
$(\Delta^{(r)}_j)^{\op}$ and $T^{(r)}_j$ a multiplicity-free tilting $H^{(r)}_j$-module without preprojective
indecomposable direct summands. Since $\mathcal{X}\Gamma_{B^{(r)}_j}$ contains infinitely many components
different from the preinjective component $\mathcal{Q}(B^{(r)}_j)$, we may choose a regular component $\mathcal{D}$
in $\mathcal{X}\Gamma_{B^{(r)}_j}$ different from $\mathcal{C}$. Note that $\mathcal{D}$ is of the form
$\mathbb{ZA}_{\infty}$. We know also that  the preinjective component $\mathcal{Q}(B^{(r)}_j)$ coincides with the unique
preinjective component $\mathcal{Q}(C^{(r)}_j)$ of a wild concealed quotient algebra $\mathcal{C}^{(r)}_j$ of
$B^{(r)}_j$. Then $C^{(r)}_j= \End_{\Lambda^{(r)}_j}(V^{(r)}_j)$, where $\Lambda^{(r)}_j$ is a wild
hereditary algebra and $V^{(r)}_j$ is a multiplicity-free tilting module from $\add(\mathcal{P}(\Lambda^{(r)}_j))$.
In particular, the functor $\Hom_{\Lambda^{(r)}_j}(V^{(r)}_j, -): \mod \Lambda^{(r)}_j \rightarrow \mod C^{(r)}_j$
induces an equivalence $\xymatrix{\add (\mathcal{R}(\Lambda^{(r)}_j)) \ar[r]^(.5){\sim}& \add
(\mathcal{R}(C^{(r)}_j))}$ of the categories of regular modules over $\Lambda^{(r)}_j$ and $C^{(r)}_j$.
Applying Lemma 2.1, we conclude that for any indecomposable modules $M$ and $N$ in $\mathcal{R}(C^{(r)}_j)$
there exists a positive integer $p$ such that $\Hom_{C^{(r)}_j}(M,\tau^s_{C^{(r)}_j}N)\neq 0$ for all
integers $s \geqslant p$. On the other hand, it follows from \cite[Theorem 1]{K2} that there exist indecomposable
modules $X \in \mathcal{C}$ and $Y \in \mathcal{D}$ such that the left cones $(\rightarrow X)$ of
$\mathcal{C}$ and $(\rightarrow Y)$ of $\mathcal{D}$ consist entirely of indecomposable $C^{(r)}_j$-modules
and the restriction of $\tau_{B^{(r)}_j}$ to the left cones $(\rightarrow X)$ and $(\rightarrow Y)$
coincides with $\tau_{C^{(r)}_j}$. Hence, the left cone $(\rightarrow X)$ of $\mathcal{C}$ is the left cone
$(\rightarrow \widetilde{X})$, with $\widetilde{X}=X$,  of a component $\widetilde{\mathcal{C}}$ of type
$\mathbb{ZA}_{\infty}$ of $\Gamma_{C^{(r)}_j}$, and the left cone $(\rightarrow Y)$ of $\mathcal{D}$ is the
left cone $(\rightarrow \widetilde{Y})$, with $\widetilde{Y}=Y$, of a component $\widetilde{\mathcal{D}}$ of type $\mathbb{ZA}_{\infty}$
of $\Gamma_{C^{(r)}_j}$. Observe that $\widetilde{\mathcal{C}} \neq \widetilde{\mathcal{D}}$ since
$\mathcal{C} \neq \mathcal{D}$. Then there exist positive integers $p$ and $q$ such that
$\Hom_{B^{(r)}_j}(X, \tau^p_{B^{(r)}_j}Y)= \Hom_{C^{(r)}_j}(X, \tau_{C^{(r)}_j}^p Y) \neq 0$ and
$\Hom_{B^{(r)}_j}(\tau^p_{B^{(r)}_j}Y, \tau^q_{B^{(r)}_j}X)= \Hom_{C^{(r)}_j}(\tau_{C^{(r)}_j}^p Y,
\tau_{C^{(r)}_j}^q X)\neq 0$. Therefore, we get  an external short path $X \rightarrow \tau^p_{B^{(r)}_j}Y
\rightarrow  \tau_{B^{(r)}_j}^q X$ of $\mathcal{C}$ in $\mod B^{(r)}_j$, and so in $\mod B$, because
$B^{(r)}_j$ is a quotient algebra of $B$ and there is a fully faithful embedding $\mod B^{(r)}_j \hookrightarrow
\mod B$.

We note that although the proofs in the two considered cases are similar, the applied results concerning the structure of left cones
of acyclic components in $\mathcal{Y}(T)$ and $\mathcal{X}(T)$ are different.
\end{proof}
Let $\Lambda$ be a canonical algebra in the sense of Ringel \cite{R4} (see also \cite{R2}). Then the
valued quiver $Q_{\Lambda}$ of $\Lambda$ has a unique sink and a unique source. Denote by $Q^{\ast}_{\Lambda}$ the
valued quiver obtained from $Q_{\Lambda}$  by removing the unique source of $Q_{\Lambda}$ and the arrows
attached to it. Then $\Lambda$ is said to be a {\em canonical algebra of Euclidean type} (respectively,
of {\em tubular type}, of {\em wild type}) if $Q^{\ast}_{\Lambda}$ is a Dynkin quiver (respectively, a Euclidean quiver,
a wild quiver). We refer to \cite[Theorems 3.1 and 3.2]{SY5} for the shapes of the valued quivers of
canonical algebras of Euclidean and tubular type. The general shape of the Auslander-Reiten quiver
$\Gamma_{\Lambda}$ of $\Lambda$, described in \cite[Sections 3 and 4]{R4}, is as follows:
\[\Gamma_{\Lambda}=\mathcal{P}^{\Lambda} \vee \mathcal{T}^{\Lambda} \vee \mathcal{Q}^{\Lambda},\]
where $\mathcal{P}^{\Lambda}$ is a family of components containing a unique preprojective component
$\mathcal{P}(\Lambda)$ and all indecomposable projective $\Lambda$-modules, $\mathcal{Q}^{\Lambda}$
is a family of components containing a unique preinjective component $\mathcal{Q}(\Lambda)$ and all
indecomposable injective $\Lambda$-modules, and $\mathcal{T}^{\Lambda}$ is an infinite family of
pairwise orthogonal generalized standard faithful stable tubes separating $\mathcal{P}^{\Lambda}$
from $\mathcal{Q}^{\Lambda}$, and with all but finitely many stable tubes of rank one. An algebra
$C$ of the form $\End_{\Lambda}(T)$, where $T$ is a multiplicity-free tilting module from the additive
category $\add(\mathcal{P}^{\Lambda})$ of $\mathcal{P}^{\Lambda}$ is said a {\em concealed canonical algebra} of type $\Lambda$.
More generally, an algebra $B$ of the form $\End_{\Lambda}(T)$, where $T$ is a multiplicity-free tilting
module from the additive category  $\add(\mathcal{P}^{\Lambda}\cup \mathcal{T}^{\Lambda})$ of
$\mathcal{P}^{\Lambda}\cup \mathcal{T}^{\Lambda}$ is said to be an {\em almost  concealed canonical algebra}
of type $\Lambda$.

We note the following statements:
\begin{itemize}
\item[$\bullet$] The class of concealed canonical algebras of Euclidean types coincides with the class of
concealed algebras of Euclidean types (tame concealed algebras).
\item[$\bullet$] The class of almost concealed canonical algebras of Euclidean types coincides with the class
of tilted algebras of the form $\End_H(T)$, where $H$ is a hereditary algebra of a Euclidean type and $T$
is a multiplicity-free tilting $H$-module without preinjective indecomposable direct summands.
\item[$\bullet$] The class of the opposite algebras of almost concealed canonical algebras of Euclidean types coincides
with the class of tilted algebras of the form $\End_H(T)$, where $H$ is a hereditary algebra of a Euclidean
type and $T$ is a multiplicity-free tilting $H$-module without preprojective indecomposable direct summands.
\end{itemize}
An almost concealed canonical algebra $B$ of a tubular type is called a {\em tubular algebra}. It is known
that  then the opposite algebra $B^{\op}$ of $B$ is also a tubular algebra. The shape of the Auslander-Reiten
quiver $\Gamma_B$ of a tubular algebra $B$, described by Ringel (see \cite[Chapter 5]{R2}, \cite[Sections 3 and 4]{R4}),
is as follows:
\[\Gamma_B= \mathcal{P}(B) \vee \mathcal{T}^B_0 \vee (\bigvee_{q \in \mathbb{Q}^+} \mathcal{T}^B_q) \vee
\mathcal{T}^B_{\infty} \vee \mathcal{Q}(B),\]
where $\mathbb{Q}^+$ is the set of positive rational numbers, $\mathcal{P}(B)$ is a preprojective
component with a Euclidean section, $\mathcal{Q}(B)$ is a preinjective component with a Euclidean section,
$\mathcal{T}^B_0$ is an inifnite family of pairwise orthogonal generalized standard ray tubes containing
at least one indecomposable projective $B$-module, $\mathcal{T}^B_{\infty}$ is an infinite family of pairwise
orthogonal generalized standard coray tubes containing at least one indecomposable injective $B$-module,
and each $\mathcal{T}^B_q$, for $q \in \mathbb{Q}^+$, is  an infinite family of pairwise orthogonal
generalized standard faithful stable tubes. Moreover, every component of $\Gamma_B$ has no external short path in
$\mod B$.

We will need also an analogue of Lemma 2.1 for the canonical algebras of wild type.

Let $\Lambda$ be a canonical algebra of wild type. Then it follows from \cite{HRS}, \cite{LP1} and \cite{LS}
that there exist hereditary abelian categories $\mathcal{H}(\Lambda)$ and $\mathcal{T}(\Lambda)$ such that
the following statements hold:
\begin{itemize}
\item[$\bullet$] The bounded derived category $D^b(\mod \Lambda)$ of $\Lambda$ has a decomposition
\[D^b(\mod\Lambda) = \bigvee_{m \in \mathbb{Z}}(\mathcal{H}(\Lambda)[m] \vee \mathcal{T}(\Lambda)[m])
= D^b(\mathcal{H}(\Lambda)^{\ast})\]
with $\mathcal{H}(\Lambda)=\mathcal{H}(\Lambda)[0]$, $\mathcal{T}(\Lambda)=\mathcal{T}(\Lambda)[0]$,
and $\mathcal{H}(\Lambda)^{\ast}= \add (\mathcal{H}(\Lambda) \vee \mathcal{T}(\Lambda))$.
\item[$\bullet$] $\mathcal{H}(\Lambda)$ is the additive category of infinitely many components of the form
$\mathbb{ZA}_{\infty}$.
\item[$\bullet$] $\mathcal{T}(\Lambda)$ is the additive category of an infinite family of pairwise orthogonal
generalized standard stable tubes.
\item[$\bullet$] Every concealed canonical algebra $C$ of type $\Lambda$ is of the form
$\End_{\mathcal{H}(\Lambda)}(T)$ for a tilting object  $T$ in $\mathcal{H}(\Lambda)$.
\item[$\bullet$]  Every almost concealed canonical algebra $B$ of type $\Lambda$ is of the form
$\End_{\mathcal{H}(\Lambda)^{\ast}}(T^{\ast})$ for a tilting object $T^{\ast}$ in $\mathcal{H}(\Lambda)^{\ast}$.
\end{itemize}
Then the following lemma is a direct  consequence of \cite[Theorem 2.7]{LP1}.
\begin{lemma}
Let $\Lambda$ be a concealed canonical algebra of wild type, and $X,Y$ be two indecomposable objects in
$\mathcal{H}(\Lambda)$. Then there is a positive integer $m$ such that
$\Hom_{\mathcal{H}(\Lambda)}(X, \tau^r_{\mathcal{H}(\Lambda)}Y) \neq 0$ for all $r \geqslant m$.
\end{lemma}
We will prove now Theorem \ref{thm6} in the canonical case.
\begin{proposition}
Let $B$ be a quasi-tilted algebra of canonical type and $\mathcal{C}$ a component of $\Gamma_B$.
The following statements are equivalent:
\begin{itemize}
\item[(i)] $\mathcal{C}$ has no external short path.
\item[(ii)] $\mathcal{C}$ is almost periodic.
\item[(iii)] $\mathcal{C}$ is generalized standard.
\item[(iv)] $\mathcal{C}$ is either a preprojective component, a preinjective component, a ray tube, or
a coray tube.
\end{itemize}
\end{proposition}
\begin{proof}
We start with the general view on the module category $\mod B$ due to results established in
\cite[Sections 3 and 4]{LS} and \cite{M}. There are an almost concealed canonical algebra
$B^{(r)}= \End_{\Lambda^{(r)}}(T^{(r)})$ and the opposite algebra $B^{(l)}$ of an almost
concealed canonical algebra $ \End_{\Lambda^{(l)}}(T^{(l)})$,
for canonical algebras $\Lambda^{(l)}$ and $\Lambda^{(r)}$ and tilting modules $T^{(l)} \in
\add (\mathcal{P}^{\Lambda^{(l)}} \vee \mathcal{T}^{\Lambda^{(l)}})$ and $T^{(r)} \in
\add (\mathcal{P}^{\Lambda^{(r)}} \vee \mathcal{T}^{\Lambda^{(r)}})$, such that $B^{(l)}$ and $B^{(r)}$ are quotient
algebras of $B$. Moreover, the Auslander-Reiten quiver  $\Gamma_B$ of $B$ has the disjoint union form
\[\Gamma_B= \mathcal{P}^B \vee \mathcal{T}^B \vee \mathcal{Q}^B,\]
where
\begin{itemize}
\item[(a)] $\mathcal{T}^B$ is a family of pairwise orthogonal generalized standard semi-regular tubes
(ray  and coray tubes) separating $\mathcal{P}^B$ from $\mathcal{Q}^B$.
\item[(b)] $\mathcal{P}^B=\mathcal{P}^{B^{(l)}}$ is a family  of components consisting entirely of
indecomposable $B^{(l)}$-modules and containing all indecomposable projective $B$-modules which are not
in $\mathcal{T}^B$.
\item[(c)] $\mathcal{P}^B$ contains a unique preprojective component $\mathcal{P}(B)$ of $\Gamma_B$,
and $\mathcal{P}(B)=\mathcal{P}(B^{(l)})$ coincides with a unique preprojective component $\mathcal{P}(C^{(l)})$
of the Auslander-Reiten quiver $\Gamma_{C^{(l)}}$ of a connected concealed quotient algebra $C^{(l)}$
of $B^{(l)}$, and hence of $B$.
\item[(d)] $\mathcal{Q}^B=\mathcal{Q}^{B^{(r)}}$ is a family of components consisting entirely of
indecomposable $B^{(r)}$-modules and containing all indecomposable injective $B$-modules which are not
in $\mathcal{T}^B$.
\item[(e)] $\mathcal{Q}^B$ contains a unique preinjective  component $\mathcal{Q}(B)$ of $\Gamma_B$,
and $\mathcal{Q}(B) = \mathcal{Q}(B^{(r)})$ coincides with a unique preinjective component $\mathcal{Q}(C^{(r)})$
of the Auslander-Reiten quiver  $\Gamma_{C^{(r)}}$ of a connected concealed quotient algebra $C^{(r)}$
of $B^{(r)}$, and hence of $B$.
\end{itemize}
Moreover, we have the following  description of components of $\Gamma_B$ contained in the parts $\mathcal{P}^B$
and $\mathcal{Q}^B$:
\begin{itemize}
\item[(1)] If $B^{(l)}$ is of Euclidean type, then $\mathcal{P}^B=\mathcal{P}(B)$.
\item[(2)] If $B^{(l)}$ is of tubular type, then
\[ \mathcal{P}^B = \mathcal{P}(B^{(l)}) \vee \mathcal{T}_0^{B^{(l)}} \vee (\bigvee_{q \in \mathbb{Q}^+}
\mathcal{T}^{B^{(l)}}_q).\]
\item[(3)] If $B^{(l)}$ is of wild canonical type, then every component of $\mathcal{P}^B$ different from
the preprojective component $\mathcal{P}(B)=\mathcal{P}(B^{(l)})=\mathcal{P}(C^{(l)})$ is obtained from
a component of the form $\mathbb{ZA}_{\infty}$ by a finite number (possibly empty) of ray insertions,
and there are infinitely many components of this type in $\mathcal{P}^B$.
\item[(4)] If $B^{(r)}$ is of Euclidean type , then $\mathcal{Q}^B=\mathcal{Q}(B)$.
\item[(5)] If $B^{(r)}$ is of tubular type, then
\[\mathcal{Q}^B= (\bigvee_{q \in \mathbb{Q}^+} \mathcal{T}^{B^{(r)}}_q) \vee \mathcal{T}^{B^{(r)}}_{\infty}
\vee \mathcal{Q}(B^{(r)}).\]
\item[(6)] If $B^{(r)}$ is of wild canonical type, then every component of $\mathcal{Q}^B$ different from the
preinjective component $\mathcal{Q}(B)=\mathcal{Q}(B^{(r)})= \mathcal{Q}(C^{(r)})$ is obtained from the component
of the form $\mathbb{ZA}_{\infty}$ by a finite number (possibly empty) of coray insertions,
and there are infinitely many components of this type in $\mathcal{Q}^B$.
\end{itemize}
It follows from the above facts that the preprojective components, preinjective components, ray tubes,
and coray tubes of $\Gamma_B$ are generalized standard, without external short paths, and clearly are almost periodic.
On the other hand, the components of $\Gamma_B$ obtained from the components of the form $\mathbb{ZA}_{\infty}$
by ray insertions or coray insertions are not almost  periodic, and hence are not generalized standard,
again by \cite[Theorem 2.3]{S2}. Therefore, it remains to show that all these components have external
short paths. We have two cases to consider.

Assume $\mathcal{C}$ is an acyclic component of $\Gamma_B$ with infinitely many $\tau_B$-orbits contained in the part
$\mathcal{P}^B$. Then, applying (1)-(3), we conclude that $B^{(l)}$ is of wild canonical type and
$\mathcal{C}$ is obtained from a component of the form $\mathbb{ZA}_{\infty}$ by a finite number
(possibly empty) of ray insertions. Since $\mathcal{P}^B$ contains infinitely many components we may also
choose in $\mathcal{P}^B$ a regular component $\mathcal{D}$ (of the form $\mathbb{ZA}_{\infty}$)
different from $\mathcal{C}$. Then it follows from the dual of \cite[Theorem 3.4]{M} that there are
components $\widetilde{\mathcal{C}}$ and $\widetilde{\mathcal{D}}$ of the form $\mathbb{ZA}_{\infty}$
in the Auslander-Reiten quiver $\Gamma_{\mathcal{H}(\Lambda^{(l)})}$ of the hereditary abelian category
$\mathcal{H}(\Lambda^{(l)})$, associated to the wild canonical algebra $\Lambda^{(l)}$, indecomposable
modules $X \in \mathcal{C}$ and $Y \in \mathcal{D}$, indecomposable objects $\widetilde{X} \in
\widetilde{\mathcal{C}}$ and $\widetilde{Y} \in \widetilde{\mathcal{D}}$,
and a functor $F^{(l)}: \mathcal{H}(\Lambda^{(l)}) \rightarrow
\mod B^{(l)}$ which induces equivalences of the additive categories of the left cones
$\xymatrix{F^{(l)}:\add(\rightarrow \widetilde{X})\ar[r]^(.55){\sim}& \add(\rightarrow X)}$ and
$\xymatrix{F^{(l)}:\add(\rightarrow \widetilde{Y})\ar[r]^(.55){\sim}& \add(\rightarrow Y)}$, such that
$F^{(l)}(\tau_{\mathcal{H}(\Lambda^{(l)})}M)=\tau_{B^{(l)}}F^{(l)}(M)$ and $F^{(l)}
(\tau_{\mathcal{H}(\Lambda^{(l)})}N)= \tau_{B^{(l)}}F^{(l)}(N)$ for all indecomposable objects
$M \in (\rightarrow \widetilde{X})$ and $N \in (\rightarrow\widetilde{Y})$. Applying Lemma 2.3, we
obtain that there exist positive integers $r$ and $s$ such that $\Hom_{\mathcal{H}(\Lambda^{(l)})}
(\widetilde{X}, \tau^r_{\mathcal{H}(\Lambda^{(l)})}\widetilde{Y}) \neq 0$ and
$\Hom_{\mathcal{H}(\Lambda^{(l)})}(\tau^r_{\mathcal{H}(\Lambda^{(l)})}\widetilde{Y},
\tau^s_{\mathcal{H}(\Lambda^{(l)})}\widetilde{X}) \neq 0$. Hence we get  $\Hom_{B^{(l)}}
(X, \tau^r_{B^{(l)}}Y) \neq 0$ and $\Hom_{B^{(l)}}(\tau^r_{B^{(l)}}Y, \tau^s_{B^{(l)}}X) \neq 0$,
and consequently an external short path $X \rightarrow \tau^r_{B^{(l)}}Y \rightarrow \tau^s_{B^{(l)}}X $
of $\mathcal{C}$ in $\mod B^{(l)}$, and hence in $\mod B$, because $B^{(l)}$ is a quotient algebra of $B$
and there is a fully faithful embedding $\mod B^{(l)}\hookrightarrow \mod B$.

Assume $\mathcal{C}$ is an acyclic component of $\Gamma_B$ with infinitely  many $\tau_B$-orbits contained in the part
$\mathcal{Q}^B$. Then, applying (4)-(6), we conclude that $B^{(r)}$ is of wild canonical type and $\mathcal{C}$
is obtained from a component of the form $\mathbb{ZA}_{\infty}$ by a finite number (possibly  empty) of
coray insertions. Since $\mathcal{Q}^B$ contains  infinitely  many components, we may choose in $\mathcal{Q}^B$
a regular component $\mathcal{D}$ (of the form $\mathbb{ZA}_{\infty}$) different from $\mathcal{C}$.
Moreover, by \cite[Theorem 6.1]{M}, the connected concealed quotient algebra $C^{(r)}$ of $B^{(r)}$ such that
$\mathcal{Q}(C^{(r)})$ is a unique preinjective component of $\Gamma_B$, is a wild concealed algebra.
Hence, $C^{(r)}=\End_{H^{(r)}}(T^{(r)})$ for a wild hereditary algebra $H^{(r)}$ and a multiplicity-free tilting
$H^{(r)}$-module from the additive category $\add (\mathcal{Q}(H^{(r)}))$ of the preinjective component
$\mathcal{Q}(H^{(r)})$ of $\Gamma_{H^{(r)}}$. In particular, the functor
$\Ext^1_{H^{(r)}}(T^{(r)},-): \mod H^{(r)}
\rightarrow \mod C^{(r)}$ induces an equivalence $\xymatrix{\add(\mathcal{R}(H^{(r)}))
\ar[r]^(.49){\sim}& \add(\mathcal{R}(C^{(r)}))}$ of the categories of regular modules over $H^{(r)}$ and $C^{(r)}$.
Applying Lemma 2.1, we conclude that for any indecomposable modules $M$ and $N$ in $\mathcal{R}(C^{(r)})$
there exists a positive integer $p$ such that $\Hom_{C^{(r)}}(M, \tau^s_{C^{(r)}}N) \neq 0$ for all integers $s \geqslant p$.
On the other hand, it follows from \cite[Theorem 6.4]{M} that there exist indecomposable modules $X \in
\mathcal{C}$ and $Y \in \mathcal{D}$ such that the left cones $(\rightarrow X)$ of $\mathcal{C}$ and
$(\rightarrow Y)$ of $\mathcal{D}$ consist entirely of indecomposable $C^{(r)}$-modules and the
restriction of $\tau_{B^{(r)}}$ to the left cones $(\rightarrow X)$ and $(\rightarrow Y)$ coincides
with $\tau_{C^{(r)}}$. Hence, the left cone $(\rightarrow X)$ of $\mathcal{C}$ is the left cone
$(\rightarrow \widetilde{X})$, with $X=\widetilde{X}$, of a component $\widetilde{\mathcal{C}}$
of type $\mathbb{ZA}_{\infty}$ of $\Gamma_{C^{(r)}}$, and the left cone $(\rightarrow Y)$ of
$\mathcal{D}$ is the left cone $(\rightarrow \widetilde{Y})$, with $\widetilde{Y}=Y$, of a component
$\widetilde{\mathcal{D}}$ of the form $\mathbb{ZA}_{\infty}$ of $\Gamma_{C^{(r)}}$. Note that
$\widetilde{\mathcal{C}}\neq \widetilde{\mathcal{D}}$ since $\mathcal{C} \neq \mathcal{D}$. Then there
exist positive integers $p$ and $q$ such that $\Hom_{B^{(r)}}(X, \tau^p_{B^{(r)}}Y) = \Hom_{C^{(r)}}
(X, \tau^p_{C^{(r)}}Y) \neq 0$ and $\Hom_{B^{(r)}}(\tau^p_{B^{(r)}}Y, \tau^q_{B^{(r)}}X) =
\Hom_{C^{(r)}}(\tau^p_{C^{(r)}}Y, \tau^q_{C^{(r)}}X) \neq 0$. Therefore, we obtain an external short path
$X \rightarrow \tau^p_{B^{(r)}}Y \rightarrow \tau^q_{B^{(r)}}X$ of $\mathcal{C}$ in $\mod B^{(r)}$,
and so in $\mod B$, because $B^{(r)}$ is a quotient algebra of $B$ and there is a fully faithful embedding
$\mod B^{(r)} \hookrightarrow \mod B$.

We note also that although the proofs in the two considered cases are similar, the applied results concerning the structure of left cones
of acyclic components in $\mathcal{P}^B$ and $\mathcal{Q}^B$ are different.
\end{proof}
The following corollary is a direct consequence of Propositions 2.2 and 2.4.
\begin{corollary}
Let $B$ be a quasi-tilted algebra. The following statements are equivalent.
\begin{itemize}
\item[(i)] No component of $\Gamma_B$ has an external short path.
\item[(ii)] The component quiver $\Sigma_B$ is acyclic.
\item[(iii)] Every component of $\Gamma_B$ is generalized standard.
\item[(iv)] $\Gamma_B$ is almost periodic.
\end{itemize}
\end{corollary}
We end this section with an example showing that an Auslander-Reiten component without external short paths is not necessarily generalized standard.
\begin{example}\label{ex2.6}
{\rm Let $K$ be an algebraically closed field, $Q$ the quiver
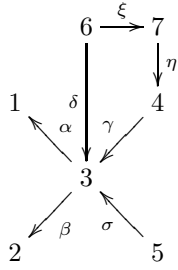
\[\xymatrix@C=16pt@R=16pt{
&6\ar[r]^{\xi}\ar[dd]_{\delta}&7\ar[d]^{\eta}\cr
1&&4\ar[ld]_{\gamma}\cr
&3\ar[lu]_{\alpha}\ar[ld]^{\beta}\cr
2&&5\ar[lu]^{\sigma}\cr
}\]
$I$ the ideal in the path algebra $KQ$ of $Q$ generated by the elements $\delta\alpha$, $\delta\beta$, $\xi\eta$ and $\eta\gamma$, and $A=KQ/I$ the
associated bound quiver algebra. We denote by $H$ the path algebra $K\Delta$ of the full subquiver $\Delta$ of $Q$ given by the vertices
$1, 2, 3, 4, 5$ and the arrows $\alpha$, $\beta$, $\gamma$, $\sigma$. Further, let $\Omega$ be the quiver obtained from $Q$ by removing the arrow $\eta$,
$J$ the ideal of the path algebra $K\Omega$ of $\Omega$ generated by $\delta\alpha$ and $\delta\beta$, and $B=K\Omega/J$ the associated bound quiver
algebra. Then $H$ is a hereditary algebra of Euclidean type $\widetilde{\mathbb{D}}_4$ whose Auslander-Reiten quiver $\Gamma_H$ consists of a preprojective
component $\mathcal{P}(H)$, a preinjective component $\mathcal{Q}(H)$, and a $\mathbb{P}_1(K)$-family
$\mathcal{T}^H=(\mathcal{T}_{\lambda}^H)_{\lambda\in\mathbb{P}_1(K)}$ of pairwise orthogonal generalized standard stable tubes, the three tubes
$\mathcal{T}_{\infty}^H$, $\mathcal{T}_{0}^H$, $\mathcal{T}_{1}^H$ of rank $2$, and the remaining tubes $\mathcal{T}_{\lambda}^H$,
$\lambda\in\mathbb{P}_1(K)\setminus\{0,1,\infty\}$, of rank $1$ (we refer to \cite[Theorem XIII.2.9]{SS1} for a detailed description of the stable tubes
of $\Gamma_H$). In particular, the simple $H$-module $S_3$ at the vertex $3$ lies on the mouth of a stable tube, say $\mathcal{T}_{1}^H$ of rank $2$.
Further, $B$ is a tubular (branch) extension (see \cite[(4.7)]{R2} or \cite[Chapters XV-XVII]{SS2}) of $H$, involving the simple module $S_3$, which is
a tilted algebra of Euclidean type $\widetilde{\mathbb{D}}_6$ such that the Auslander-Reiten quiver $\Gamma_B$ of $B$ has disjoint union form
\[ \Gamma_B = \mathcal{P}(B)\vee\mathcal{T}^B\vee\mathcal{Q}(B), \]
where $\mathcal{P}(B)=\mathcal{P}(H)$ is a unique preprojective component, $\mathcal{Q}(B)$ is a unique preinjective component containing all
indecomposable injective $B$-modules and $\mathcal{T}^B=(\mathcal{T}_{\lambda}^B)_{\lambda\in\mathbb{P}_1(K)}$ is a $\mathbb{P}_1(K)$-family
of pairwise orthogonal generalized standard ray tubes with $\mathcal{T}_{\lambda}^B=\mathcal{T}_{\lambda}^H$ for $\lambda\in\mathbb{P}_1(K)\setminus\{1\}$,
and $\mathcal{T}_{1}^B$ is a ray tube obtained from the stable tube $\mathcal{T}_{1}^H$ by insertion of two rays, containing the indecomposable projective
$B$-module $P_6$ and $P_7$ at the vertices $6$ and $7$. Moreover, we have $\Hom_B(\mathcal{T}^B,\mathcal{P}(B))=0$, $\Hom_B(\mathcal{Q}(B),\mathcal{T}^B)=0$,
$\Hom_B(\mathcal{Q}(B),\mathcal{P}(B))=0$, and $\Hom_B(\mathcal{P}(B),\mathcal{T}_{\lambda}^B)\neq 0$, $\Hom_B(\mathcal{T}_{\lambda}^B,\mathcal{Q}(B))\neq 0$
for all $\lambda\in\mathbb{P}_1(K)$. Observe also that $B$ is the quotient algebra of $A$ by the ideal generated by the coset $\eta + I$ of the arrow $\eta$,
and hence we have the fully faithful embedding $\mod B\hookrightarrow\mod A$. Then using the canonical equivalences
$\mod A\buildrel{\thicksim}\over{\hbox to 6mm{\rightarrowfill}}\rep_K(Q,I)$ and
$\mod B\buildrel{\thicksim}\over{\hbox to 6mm{\rightarrowfill}}\rep_K(\Omega,J)$ (see \cite[Theorem III.1.6]{ASS}) we easily infer that there is only one
indecomposable module in $\mod A$ which is not in $\mod B$, namely the $2$-dimensional projective-injective module $P_7=I_4$ whose socle is the simple
module $S_4$ at the vertex $4$ and whose top is the simple module $S_7$ at the vertex $7$. Therefore, the Auslander-Reiten quiver $\Gamma_A$ of $A$ has
the disjoint union form
\[ \Gamma_A = \mathcal{P}(A)\vee\mathcal{T}^A\vee\mathcal{C}, \]
where $\mathcal{P}(A)=\mathcal{P}(B)=\mathcal{P}(H)$ is a preprojective component,
$\mathcal{T}^A=(\mathcal{T}_{\lambda}^A)_{\lambda\in\mathbb{P}_1(K)\setminus\{1\}}=$  $(\mathcal{T}_{\lambda}^B)_{\lambda\in\mathbb{P}_1(K)\setminus\{1\}}$
is a family of pairwise orthogonal generalized standard stable tubes, and $\mathcal{C}$ is the component of the form below, obtained by gluing the ray
tube $\mathcal{T}_1^B$ and the preinjective component $\mathcal{Q}(B)$ by the projective-injective module $P_7=I_4$,


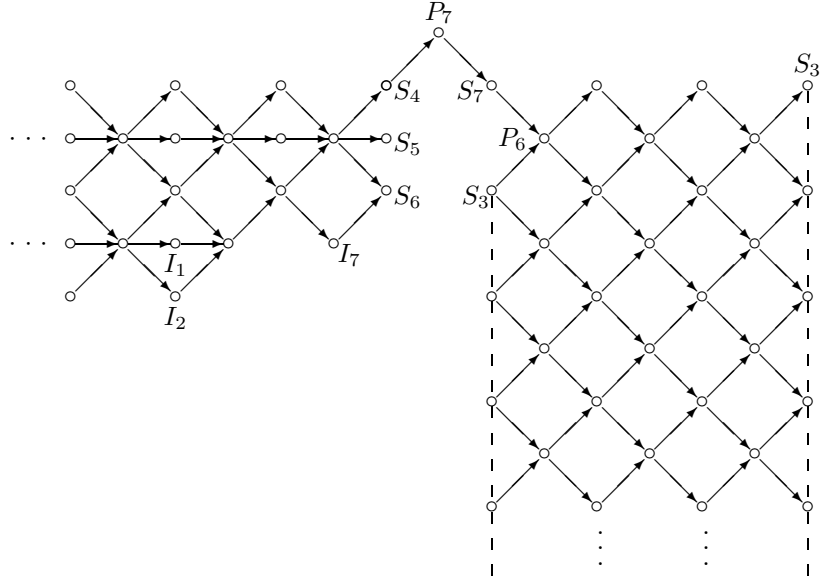
\begin{center}
\setlength{\unitlength}{0.7mm}
\begin{picture}(100,107)(-50,-100)


\put(0,0){\makebox(0,0){$\circ$}}
\put(-10,-10){\makebox(0,0){$\circ$}}
\put(-10,-10){\makebox(0,0){$\circ$}}
\put(-10,-20){\makebox(0,0){$\circ$}}
\put(-10,-30){\makebox(0,0){$\circ$}}
\put(-20,-20){\makebox(0,0){$\circ$}}
\put(-20,-40){\makebox(0,0){$\circ$}}
\put(-30,-10){\makebox(0,0){$\circ$}}
\put(-30,-20){\makebox(0,0){$\circ$}}
\put(-30,-30){\makebox(0,0){$\circ$}}
\put(-40,-20){\makebox(0,0){$\circ$}}
\put(-40,-40){\makebox(0,0){$\circ$}}
\put(-50,-10){\makebox(0,0){$\circ$}}
\put(-50,-20){\makebox(0,0){$\circ$}}
\put(-50,-30){\makebox(0,0){$\circ$}}
\put(-50,-40){\makebox(0,0){$\circ$}}
\put(-50,-50){\makebox(0,0){$\circ$}}
\put(-60,-20){\makebox(0,0){$\circ$}}
\put(-60,-40){\makebox(0,0){$\circ$}}
\put(-70,-10){\makebox(0,0){$\circ$}}
\put(-70,-20){\makebox(0,0){$\circ$}}
\put(-70,-30){\makebox(0,0){$\circ$}}
\put(-70,-40){\makebox(0,0){$\circ$}}
\put(-70,-50){\makebox(0,0){$\circ$}}


\multiput(-49,-49)(10,10){5}{\vector(1,1){8}}
\multiput(-69,-49)(10,10){4}{\vector(1,1){8}}
\multiput(-69,-29)(10,10){2}{\vector(1,1){8}}
\multiput(-19,-39)(10,10){1}{\vector(1,1){8}}

\multiput(-69,-31)(10,-10){2}{\vector(1,-1){8}}
\multiput(-69,-11)(10,-10){3}{\vector(1,-1){8}}
\multiput(-49,-11)(10,-10){3}{\vector(1,-1){8}}
\multiput(-29,-11)(10,-10){2}{\vector(1,-1){8}}

\multiput(-69,-20)(10,0){6}{\vector(1,0){8}}
\multiput(-69,-40)(10,0){3}{\vector(1,0){8}}


\put(10,-10){\makebox(0,0){$\circ$}}
\put(30,-10){\makebox(0,0){$\circ$}}
\put(50,-10){\makebox(0,0){$\circ$}}
\put(70,-10){\makebox(0,0){$\circ$}}
\put(10,-30){\makebox(0,0){$\circ$}}
\put(30,-30){\makebox(0,0){$\circ$}}
\put(50,-30){\makebox(0,0){$\circ$}}
\put(70,-30){\makebox(0,0){$\circ$}}
\put(10,-50){\makebox(0,0){$\circ$}}
\put(30,-50){\makebox(0,0){$\circ$}}
\put(50,-50){\makebox(0,0){$\circ$}}
\put(70,-50){\makebox(0,0){$\circ$}}
\put(10,-70){\makebox(0,0){$\circ$}}
\put(30,-70){\makebox(0,0){$\circ$}}
\put(50,-70){\makebox(0,0){$\circ$}}
\put(70,-70){\makebox(0,0){$\circ$}}

\put(20,-20){\makebox(0,0){$\circ$}}
\put(20,-40){\makebox(0,0){$\circ$}}
\put(20,-60){\makebox(0,0){$\circ$}}
\put(40,-20){\makebox(0,0){$\circ$}}
\put(40,-40){\makebox(0,0){$\circ$}}
\put(40,-60){\makebox(0,0){$\circ$}}
\put(60,-20){\makebox(0,0){$\circ$}}
\put(60,-40){\makebox(0,0){$\circ$}}
\put(60,-60){\makebox(0,0){$\circ$}}

\put(10,-90){\makebox(0,0){$\circ$}}
\put(30,-90){\makebox(0,0){$\circ$}}
\put(50,-90){\makebox(0,0){$\circ$}}
\put(70,-90){\makebox(0,0){$\circ$}}
\put(20,-80){\makebox(0,0){$\circ$}}
\put(40,-80){\makebox(0,0){$\circ$}}
\put(60,-80){\makebox(0,0){$\circ$}}


\multiput(1,-1)(10,-10){7}{\vector(1,-1){8}}
\multiput(31,-11)(10,-10){4}{\vector(1,-1){8}}
\multiput(51,-11)(10,-10){2}{\vector(1,-1){8}}
\multiput(11,-31)(10,-10){6}{\vector(1,-1){8}}
\multiput(11,-51)(10,-10){4}{\vector(1,-1){8}}
\multiput(11,-71)(10,-10){2}{\vector(1,-1){8}}

\multiput(11,-29)(10,10){2}{\vector(1,1){8}}
\multiput(11,-49)(10,10){4}{\vector(1,1){8}}
\multiput(11,-69)(10,10){6}{\vector(1,1){8}}

\multiput(11,-89)(10,10){6}{\vector(1,1){8}}

\multiput(31,-89)(10,10){4}{\vector(1,1){8}}
\multiput(51,-89)(10,10){2}{\vector(1,1){8}}


\multiput(10,-31)(0,-5){15}{\line(0,-1){2}}
\multiput(70,-11)(0,-5){19}{\line(0,-1){2}}

\multiput(-75,-20)(-3,0){3}{\makebox(0,0){$\cdot$}}
\multiput(-75,-40)(-3,0){3}{\makebox(0,0){$\cdot$}}

\multiput(30,-95)(0,-3){3}{.}
\multiput(50,-95)(0,-3){3}{.}


\put(0,4){\makebox(0,0){$P_7$}}
\put(-6,-11){\makebox(0,0){$S_4$}}
\put(-6,-21){\makebox(0,0){$S_5$}}
\put(-6,-31){\makebox(0,0){$S_6$}}
\put(6,-11){\makebox(0,0){$S_7$}}
\put(-17,-42){\makebox(0,0){$I_7$}}
\put(-50,-44){\makebox(0,0){$I_1$}}
\put(-50,-54){\makebox(0,0){$I_2$}}
\put(14,-20){\makebox(0,0){$P_6$}}
\put(7,-31){\makebox(0,0){$S_3$}}
\put(70,-6){\makebox(0,0){$S_3$}}

\end{picture}
\end{center}
\vspace{0,7cm}
where the modules along the vertical dashed lines have to be identified. We note that $\Hom_A(\mathcal{C},\mathcal{D})=0$ for any component $\mathcal{D}$
of $\Gamma_A$ different from $\mathcal{C}$, and hence $\mathcal{C}$ has no external short path in $\mod A$. On the other hand, the canonical monomorphism
from the simple $A$-module $S_7$ to its injective envelope $I_7$ in $\mod A$ belongs to $\rad^{\infty}_A(S_7,I_7)$, and so the component $\mathcal{C}$ is
not generalized standard.
}
\end{example}

\section{Proofs of Theorems \ref{thm1} and \ref{thm2}}

\noindent The aim of this section is to provide the proof of Theorem \ref{thm1}. Observe that Theorem \ref{thm2} follows from Theorem \ref{thm1} applied
to the opposite algebra $A^{\op}$ of $A$.

Let $A$ be an algebra. Following \cite{AR} a module $M$ in $\mod A$ is said to be the \emph{middle of a
short chain} if there is some indecomposable module $X$ in $\mod A$ with $\Hom_A(X,M) \neq 0$ and
$\Hom_A(M, \tau_AX)\neq 0$. We note that if $M$ and $N$ are indecomposable modules in $\mod A$ with the
same composition factors and $M$ is not the middle of a short chain, then $M$ and $N$ are isomorphic
(see \cite{AR}, \cite{RSS1}).

The following lemma follows from \cite[Lemma 1.3]{RS1} (see also \cite[Theorem 1.6]{RSS1}).

\begin{lemma}
Let $A$ be an artin algebra, $M$ a module in $\mod A$ and $X$ an indecomposable module in $\mod A$
which is not isomorphic to  a direct summand of $M$. Assume that $\Hom_A(X,M)\neq 0$ and
$\Hom_A(M, \tau_AX) \neq 0$. Then there is a short path $Y \rightarrow V \rightarrow Z$
in $\mod A$, where $Y$ and $Z$ are indecomposable direct  summands of $M$, and $V=X$ or
$V$ is an indecomposable direct  summand of the middle term $E$ of an Auslander-Reiten sequence
$0 \rightarrow \tau_AX \rightarrow E \rightarrow X \rightarrow 0$ in $\mod A$.
\end{lemma}
The next proposition will reduce the proof of Theorem \ref{thm1} to Theorem \ref{thm6}.
\begin{proposition}
Let $A$ be an algebra, $\mathcal{C}$ a semi-regular component of $\Gamma_A$ without
external short paths, and $B=A/ \ann_A(\mathcal{C})$. Then $B$ is a quasi-tilted algebra.
\end{proposition}

\begin{proof}
We may assume (by duality) that $\mathcal{C}$ is without projective modules.
Since $\mathcal{C}$ is a component of $\Gamma_B$ with $\ann_B (\mathcal{C})=0$, we may also assume that
$\ann_A (\mathcal{C})=0$. We choose a module $M$ in the additive category  $\add (\mathcal{C})$
of $\mathcal{C}$ such that $\ann_A(M)=\ann_A (\mathcal{C})=0$ (see \cite[Lemma 1.1]{RS1}). Then there
are a monomorphism $A \rightarrow M^r$ and an epimorphism $M^s \rightarrow D(A)$ in $\mod A$,
for some positive integers $r$ and $s$ (see \cite[Lemma VI.2.2]{ASS}).

We prove first that $\gldim A \leqslant 2$. Take an indecomposable projective module $P$ in $\mod A$
and an indecomposable direct summand $X$ of the radical $\rad P$ of $P$. Observe that $X$ is not in
$\mathcal{C}$ since there is an irreducible homomorphism $X \rightarrow P$ and $P$ is not  in
$\mathcal{C}$. On the other hand, we have $\Hom_A(X,M)\neq 0$ because there is a monomorphism
$A \rightarrow M^r$ and $P$ is a direct summand of $A$. We claim that $\pd_AX\leqslant 1$.
Assume $\pd_AX\geqslant 2$. Then we have $\Hom_A(D(A), \tau_AX) \neq 0$ (see \cite[Lemma IV.2.7]{ASS}
or \cite[(2.4)]{R2}), and hence $\Hom_A(M,\tau_AX)\neq 0$, because  there is an epimorphism
$M^s \rightarrow D(A)$ in $\mod A$. Thus $M$ is the middle of a short chain $X \rightarrow M
\rightarrow \tau_AX$ with $M$ in $\add(\mathcal{C})$ and $X, \tau_AX$ not in $\mathcal{C}$.
Applying Lemma 3.1, we conclude that there is a short path $Y \rightarrow V \rightarrow  Z$,
where $Y$ and $Z$ are indecomposable direct summands of $M$, and $V=X$ or $V$ is an indecomposable
direct  summand of the middle term $E$ of an Auslander-Reiten sequence
$$0 \rightarrow \tau_AX \rightarrow E \rightarrow X \rightarrow 0$$
in $\mod A$. Since $X$ is not in $\mathcal{C}$, $V$ is also not in $\mathcal{C}$, and so $Y \rightarrow V
\rightarrow Z$ is an external short path of $\mathcal{C}$, a contradiction. Hence, indeed
$\pd_AX\leqslant 1$. This shows that  $\pd_A \rad P \leqslant 1$, and consequently $\gldim A \leqslant 2$.

Let $N$ be an indecomposable module in $\mod A$. We claim that $\pd_AN\leqslant 1$ or $\id_AN\leqslant 1$.
We have two cases to consider. Assume first  that  $N$ belongs to $\mathcal{C}$. We prove that then
$\id_AN\leqslant 1$. Suppose $\id_AN\geqslant 2$. Then $\Hom_A(\tau^{-}_AN,A)\neq 0$ (see again
\cite[Lemma IV.2.7]{ASS} or \cite[(2.4)]{R2}), and so $\Hom_A(\tau^{-}_AN,P')\neq 0$ for an indecomposable
projective right $A$-module $P'$.  Clearly, we have also $\Hom_A(P',Z)\neq 0$ for an indecomposable direct
summand $Z$  of $M$, because there is a monomorphism $A \rightarrow M^r$. Therefore, since
$\mathcal{C}$ is without projective modules, we obtain a short path
$\tau^{-}_AN \rightarrow P' \rightarrow Z$ in $\mod A$ with $\tau^{-}_AN$ and $Z$ in $\mathcal{C}$ and $P'$
not in $\mathcal{C}$, and so  an external short path of $\mathcal{C}$, a contradiction. Hence, indeed
$\id_AN\leqslant 1$.

Assume now that $N$ is not in $\mathcal{C}$ and $\id_AN\geqslant 2$. We claim that then $\pd_AN\leqslant 1$.
We show first that $\pd_A\tau^{-}_AN\leqslant 1$. Assume $\pd_A\tau^{-}_AN\geqslant 2$. Then we  have
$\Hom_A(D(A),N)=\Hom_A(D(A), \tau_A(\tau^{-}_AN))\neq 0$, and consequently $\Hom_A(M,N)\neq 0$, because
there is an epimorphism $M^s \rightarrow D(A)$ in $\mod A$. On the other hand, the assumption $\id_AN\geqslant 2$
gives $\Hom_A(\tau^{-}_AN,A) \neq 0$. This implies $\Hom_A(\tau^{-}_AN,M)\neq 0$, because there is a monomorphism
$A \rightarrow M^r$ in $\mod A$. Observe also that $N$ and $\tau^{-}_AN$ are not  isomorphic to a direct summand of $M$,
since $N$ and $\tau^{-}_AN$ are not in $\mathcal{C}$. Therefore, applying Lemma 3.1 to the short chain
$\tau^{-}_AN \rightarrow M \rightarrow N$, we conclude that there is in $\mod A$ an external short path
$U \rightarrow V \rightarrow W$ of $\mathcal{C}$, where $U$  and $W$ are indecomposable direct summands of $M$, and
$V=\tau^{-}_AN$ or $V$ is an indecomposable direct summand of the middle term $F$ of an Auslander-Reiten
sequence
$$0 \rightarrow N \rightarrow F \rightarrow \tau^{-}_AN \rightarrow 0$$
in $\mod A$, a contradiction. Hence, indeed $\pd_A\tau^{-}_AN\leqslant 1$.
Take now an indecomposable direct summand $L$ of the middle term $F$ of the above Auslander-Reiten sequence.
We claim that $\pd_AL\leqslant 1$. Choose an irreducible homomorphism $f: L \rightarrow \tau^{-}_AN$ in
$\mod A$. By general theory we know that  $f$ is either a proper monomorphism or a proper epimorphism.
Hence we have two cases to consider.

Assume $f$ is a monomorphism. Then we have in $\mod A$ a short  exact sequence
$$0 \rightarrow L \rightarrow \tau^{-}_AN \rightarrow R \rightarrow 0,$$
and hence an exact sequence of functors
$$\Ext^2_A(\tau^{-}_AN, -) \rightarrow \Ext^2_A(L, -) \rightarrow \Ext^3_A(R, -)$$
on $\mod A$. Since $\pd_A \tau^{-}_AN\leqslant 1$ and $\gldim A\leqslant 2$, we have
$\Ext^2_A(\tau^{-}_AN, -) =0$ and $\Ext^3_A(R, -)=0$, which leads to $\Ext^2_A(L, -)=0$,
or equivalently, to $\pd_AL\leqslant 1$.

Assume $f$ is an epimorphism. Then $\Hom_A(\tau^{-}_AN,M) \neq 0$ forces $\Hom_A(L,M) \neq 0$.
Assume $\pd_AL\geqslant 2$. Then $\Hom_A(D(A), \tau_AL)\neq 0$, and hence $\Hom_A(M, \tau_AL) \neq 0$,
because there is an epimorphism $M^s \rightarrow D(A)$ in $\mod A$. Therefore, $M$ is the middle term of
a short chain $L \rightarrow M \rightarrow \tau_AL$ with $L$ and $\tau_AL$ not in $\mathcal{C}$,
since $N$ is not in $\mathcal{C}$. Applying Lemma 3.1, we conclude that  there is an external short  path
of $\mathcal{C}$ of the form $S \rightarrow V \rightarrow T$, where $S$ and $T$ are indecomposable direct
summands of $M$, and $V=L$ or $V$ is an indecomposable direct summand of the middle term $G$ of an Auslander-Reiten
sequence
$$0 \rightarrow \tau_AL \rightarrow G \rightarrow L \rightarrow 0$$
in $\mod A$, a contradiction. Hence we obtain that $\pd_AL\leqslant 1$. Summing up, we have in $\mod A$
a short  exact sequence
$$0 \rightarrow N \rightarrow F \rightarrow \tau^{-}_AN \rightarrow 0$$
with $\pd_AF\leqslant 1$. Since $\gldim A \leqslant 2$, we have an epimorphism of functors
$\Ext^2_A(F, -) \rightarrow \Ext^2_A(N, -)$, and hence $\Ext^2_A(F, -)=0$ forces $\Ext^2_A(N, -)=0$.
This proves that $\pd_AN\leqslant 1$.

Therefore, $A$  is a quasi-tilted algebra.
\end{proof}
We complete now the proof of Theorem \ref{thm1}.\\
Let $A$ be an algebra, $\mathcal{C}$ a component of $\Gamma_A$ without projective modules and external
short paths, and $B=A/\ann_A(\mathcal{C})$. It follows from Proposition 3.2 that then $B$ is a quasi-tilted
algebra. Moreover, $\mathcal{C}$ is a faithful component of $\Gamma_B$. We have two cases to consider.

Assume $\mathcal{C}$ is acyclic. Then it follows from Propositions 2.2 and 2.4 (and their proofs) that
$B$ is a tilted algebra $\End_{H}(T)$, for a hereditary algebra $H$ and a multiplicity-free tilting $H$-module
$T$, and $\mathcal{C}$ is the connecting component $\mathcal{C}_T$ of $\Gamma_B$ determined by $T$. Further, since
$\mathcal{C}=\mathcal{C}_T$ is without projective modules, we conclude also that $T$ has no non-zero preinjective
direct summands. We note that $\mathcal{C}=\mathcal{C}_T$ is a preinjective component if and only if
$T$ is a preprojective tilting $H$-module, or equivalently, $B$ is a concealed algebra. Clearly,
the component $\mathcal{C}=\mathcal{C}_T$ is regular  if and only if $T$ is a regular  tilting $H$-module.

Assume $\mathcal{C}$ contains an oriented cycle. Since $\mathcal{C}$ is without projective modules, it follows
from the general result of Liu \cite[Theorem 2.5]{L2} that $\mathcal{C}$ is a coray tube. Hence $B$ is
a quasi-tilted algebra with a faithful coray tube $\mathcal{C}$. Applying Propositions 2.2 and 2.4
(and their proofs) and \cite[Theorem 3.4]{LS} we infer that $B$ is the opposite algebra of an almost concealed
canonical algebra and $\mathcal{C}$ is a coray tube of a separating family $\mathcal{T}^B$ of coray tubes
of $\Gamma_B$. We note that, if $\mathcal{C}$ contains an injective module, then the remaining coray
tubes of $\mathcal{T}^B$ are stable tubes. On the other hand, $\mathcal{C}$ is a regular component
if and only if the algebra $B=A/\ann_A(\mathcal{C})$ is a concealed canonical algebra.

\section{Proof of Theorem \ref{thm10}}

\noindent Let $B$ be an algebra and $1_B = e_1+\ldots + e_n$ a decomposition of the identity $1_B$ of $B$ into a sum of pairwise orthogonal
primitive idempotents.
The {\em repetitive category} of $B$ is the self-injective locally bounded $K$-category $\widehat{B}$ with the objects $e_{m,i}$, $m\in\Bbb Z$,
$i\in\{1,\ldots,n\}$, the morphism $K$-modules defined as follows
\[\widehat{B}(e_{m,i},e_{r,j})=
\begin{cases}
e_jBe_i,  & r=m\cr
D(e_iBe_j),  & r=m+1\cr
0,  & {\rm otherwise},\cr
\end{cases}
\]
and the composition of morphisms in $\widehat{B}$ is given by the $B$-$B$-bimodule structures on $B$ and $D(B)$. We denote by $\nu_{\widehat{B}}$
the {\em Nakayama automorphism} of $\widehat{B}$ defined by $\nu_{\widehat{B}}(e_{m,i})=e_{m+1,i}$ for all $m\in\Bbb Z$, $i\in\{1,\ldots,n\}$. An automorphism
$\varphi$ of the $K$-category $\widehat{B}$ is said to be:
\begin{itemize}
\item[$\bullet$] {\em positive} if, for each pair $(m,i)\in{\Bbb Z}\times\{1,\ldots,n\}$, we have $\varphi(e_{m,i})=e_{p,j}$ for some $p\geq m$ and some
$j\in\{1,\ldots,n\}$;
\item[$\bullet$] {\em rigid} if, for each pair $(m,i)\in{\Bbb Z}\times\{1,\ldots,n\}$, we have $\varphi(e_{m,i})=e_{m,j}$ for some $j\in\{1,\ldots,n\}$;
\item[$\bullet$] {\em strictly positive} if it is positive but not rigid.
\end{itemize}
A group $G$ of automorphisms of the $K$-category $\widehat{B}$ is said to be {\em admissible} if $G$ acts freely on the set of objects of $\widehat{B}$
and has finitely many orbits. Following Gabriel \cite{Ga} we may then consider the orbit bounded $K$-category $\widehat{B}/G$, where the objects are the
$G$-orbits of objects of $\widehat{B}$, and hence the basic connected self-injective artin algebra $\bigoplus(\widehat{B}/G)$ given by the direct sum
of all morphism $K$-modules in $\widehat{B}/G$. We will identify $\widehat{B}/G$ with $\bigoplus(\widehat{B}/G)$ and call the {\em orbit algebra} of
$\widehat{B}$ with respect to $G$. We note that the infinite cyclic group $(\nu_{\widehat{B}})$ generated by the Nakayama automorphism $\nu_{\widehat{B}}$
of $\widehat{B}$ is admissible and the orbit algebra $\widehat{B}/(\nu_{\widehat{B}})$ is isomorphic to the trivial extension $B\ltimes D(B)$ of $B$ by
$D(B)$. We refer to \cite[Theorem 5.3]{SY4} for a criterion on a self-injective algebra $A$ to be of the form $\widehat{B}/(\varphi\nu_{\widehat{B}})$
with $B$ an algebra and $\varphi$ a positive automorphism of $\widehat{B}$.

The implication (ii)$\Rightarrow$(i) in Theorem \ref{thm10} follows from the part (iii) of the following proposition.

\begin{proposition} \label{prop4.1}
Let $B$ be a tilted algebra of the form $\End_H(T)$ for a hereditary algebra $H$ and a regular tilting $H$-module $T$, $\mathcal{C}_T$ the connecting
component of $\Gamma_B$ determined by $T$, $\psi$ a positive automorphism of $\widehat{B}$ and $A=\widehat{B}/(\psi\nu_{\widehat{B}})$ the orbit algebra
of $\widehat{B}$ with respect to the infinite cyclic group generated by $\psi\nu_{\widehat{B}}$. Then the following statements hold.
\begin{itemize}
\item[(i)] $B$ is a quotient algebra of $A$.
\item[(ii)] $\mathcal{C}_T$ is a generalized standard regular component of $\Gamma_A$.
\item[(iii)] $\mathcal{C}_T$ has no external short path in $\mod A$ if and only if $\psi=\varphi\nu_{\widehat{B}}$ for a positive automorphism $\varphi$
of $\widehat{B}$.
\end{itemize}
\end{proposition}

\begin{proof}
The statement (i) follows from the definition of the orbit algebra $\widehat{B}/(\psi\nu_{\widehat{B}})$ and the positivity of $\psi$. In order to show the
statements (ii) and (iii), we present the structure of $\Gamma_{\widehat{B}}$ and $\mod\widehat{B}$ established in \cite[Section 3]{EKS}
(see also \cite[Section 2]{RSS2}).

Let $Q=Q_H$ be the valued quiver of $H$ and $\Delta=Q^{\rm op}$. Then $\Delta$ is a wild quiver with at least $3$ vertices, since we have the regular
tilting $H$-module $T$. Applying \cite[Theorem 3.5]{EKS} we conclude that there exist tilted convex subcategories $B^+=\End_H(T^+)$ and $B^-=\End_H(T^-)$
of $\widehat{B}$, for a tilting $H$-module $T^+$ without non-zero preinjective direct summands and a tilting $H$-module $T^-$ without non-zero
preprojective direct summands, such that for the shifts $B^+_{2q}=B^-_{2q}=\nu_{\widehat{B}}^q(B)$, $B^+_{2q+1}=\nu_{\widehat{B}}^q(B^+)$,
$B^-_{2q+1}=\nu_{\widehat{B}}^q(B^-)$, $q\in\mathbb{Z}$, of $B$, $B^{+}$ and $B^{-}$ inside $\widehat{B}$, the following statements hold:
\begin{itemize}
\item[(a)] The Auslander-Reiten quiver $\Gamma_{\widehat{B}}$ of $\widehat{B}$ has the disjoin union form
\[ \Gamma_{\widehat{B}} = \bigvee_{q\in\mathbb{Z}}(\mathcal{C}_q\vee \mathcal{R}_q), \]
where $\nu_{\widehat{B}}(\mathcal{C}_q)=\mathcal{C}_{q+2}$, $\nu_{\widehat{B}}(\mathcal{R}_q)=\mathcal{R}_{q+2}$,
$\Hom_{\widehat{B}}(\mathcal{R}_q,\mathcal{C}_q)=0$, for any $q\in\mathbb{Z}$, and
$\Hom_{\widehat{B}}(\mathcal{C}_p\vee\mathcal{R}_p,\mathcal{C}_q\vee\mathcal{R}_q)=0$ for all $p>q$. Moreover, each $\mathcal{C}_q$ separates
$\bigvee_{p<q}(\mathcal{C}_p\vee\mathcal{R}_p)$ from $\mathcal{R}_q\vee(\bigvee_{p>q}(\mathcal{C}_p\vee\mathcal{R}_p))$ and each $\mathcal{R}_q$
separates $\bigvee_{p<q}(\mathcal{C}_p\vee\mathcal{R}_p)\vee\mathcal{C}_q$ from $\bigvee_{p>q}(\mathcal{C}_p\vee\mathcal{R}_p)$.
\item[(b)] For each $q\in\mathbb{Z}$, $\mathcal{C}_{2q}$ is the regular connecting component $\mathcal{C}_T$ of $\Gamma_{B^+_{2q}}$ (respectively,
$\Gamma_{B^-_{2q}}$) of the form $\mathbb{Z}\Delta$ determined by the regular tilting $H$-module $T$.
\item[(c)] For each $q\in\mathbb{Z}$, $\mathcal{C}_{2q+1}$ is an acyclic component with the stable part $\mathcal{C}_{2q+1}^s$ of the form
$\mathbb{Z}\Delta$, the torsion-free part $\mathcal{Y}(T^+)\cap\mathcal{C}_{T^+}$ of the connecting component $\mathcal{C}_{T^+}$ of $\Gamma_{B^+_{2q+1}}=\Gamma_{B^+}$
determined by $T^+$ is a full translation subquiver of $\mathcal{C}_{2q+1}$ which is closed under predecessors, and the torsion part
$\mathcal{X}(T^-)\cap\mathcal{C}_{T^-}$ of the connecting component $\mathcal{C}_{T^-}$ of $\Gamma_{B^-_{2q+1}}=\Gamma_{B^-}$
determined by $T^-$ is a full translation subquiver of $\mathcal{C}_{2q+1}$ which is closed under successors.
\item[(d)] For each $q\in\mathbb{Z}$, $\mathcal{R}_{q}$ is an infinite family of components whose stable parts are of the form
$\mathbb{Z}\mathbb{A}_{\infty}$, and the simple composition factors of modules in $\mathcal{R}_q$ are simple $B^-_q$-modules or simple
$B^+_{q+1}$-modules.
\item[(e)] For each $m\in\mathbb{Z}$, the indecomposable projective-injective $\widehat{B}$-modules $P_{m+1,i} = e_{m+1,i}\widehat{B} =
\nu_{\widehat{B}}(e_{m,i}\widehat{B}) = \nu_{\widehat{B}}(P_{m,i})$, $i\in\{1,\ldots, n\}$, lie in $\mathcal{R}_{2m}\vee\mathcal{C}_{2m+1}\vee\mathcal{R}_{2m+1}$. \\

Moreover, we claim that the following fact holds. \\

\item[(f)] For each $q\in\mathbb{Z}$, the family of components $\mathcal{R}_q$ contains at least one indecomposable projective-injective
$\widehat{B}$-module.\\

Assume that $\mathcal{R}_q$, for some $q\in\mathbb{Z}$, consists entirely of regular components of the form $\mathbb{Z}\mathbb{A}_{\infty}$. Then it
follows from the above description of $\Gamma_{\widehat{B}}$ that in the Auslander-Reiten quiver $\Gamma_B$ of the tilted algebra $B=\End_H(T)$ either
all indecomposable projective $B$-modules lie in the unique preprojective component $\mathcal{P}(B)$ or all indecomposable injective $B$-modules lie in
the unique preinjective component $\mathcal{Q}(B)$. Hence, $\Gamma_B$ admits at least two components having a section of type $\Delta$, one of them
the regular connecting component $\mathcal{C}_T$. On the other hand, by a result due to Ringel \cite[Lectures 2 and 3]{R3}, if the Auslander-Reiten
quiver $\Gamma_B$ contains at least two components with complete sections then $B$ is a concealed algebra, a contradiction. \\

We note the following obvious consequence of the description of the simple composition factors of indecomposable $\widehat{B}$-modules presented above. \\

\item[(g)] For each $q\in\mathbb{Z}$, we have $\Hom_{\widehat{B}}(\mathcal{C}_{2q},\mathcal{C}_p\vee\mathcal{R}_p)=0$ for all $p\geq 2q+2$. In particular,
we obtain that $\Hom_{\widehat{B}}(\mathcal{C}_{0},\mathcal{C}_p\vee\mathcal{R}_p)=0$ for all $p\geq 2$. \\

We prove now that: \\

\item[(h)] For each $q\in\mathbb{Z}$ and all but finitely many indecomposable modules $X$ in $\mathcal{C}_{2q}$, there exists a short path in
$\mod\widehat{B}$ of the form $X\to P\to Z$ with an indecomposable projective-injective module $P$ in $\mathcal{R}_{2q+1}$ and $Z$ in $\mathcal{C}_{2q+3}$.\\

Since $\mathcal{C}_{p+2}=\nu_{\widehat{B}}(\mathcal{C}_p)$ and $\mathcal{R}_{p+2}=\nu_{\widehat{B}}(\mathcal{R}_p)$ for each $p\in\mathbb{Z}$, we may
assume that $q=0$. It is known (see \cite[Lemma 4.1]{KS} and \cite[Lemma 1.8]{RSS2}) that all but finitely many indecomposable modules in the regular
connecting component $\mathcal{C}_0=\mathcal{C}_T$ of the tilted algebra $B=\End_H(T)$ are sincere $B$-modules. Further, the simple $B$-modules are exactly
the socles of the indecomposable projective $\widehat{B}$-modules $P_{1,i}=e_{1,i}\widehat{B}$, $i\in\{1,\ldots, n\}$, which are exactly the indecomposable
projective-injective $\widehat{B}$-modules located in $\mathcal{R}_0\vee\mathcal{C}_1\vee\mathcal{R}_1$. Applying (f) we conclude that there is
$i\in\{1,\ldots, n\}$ such that $P_{1,i}$ belongs to $\mathcal{R}_1$. Then, for an arbitrary sincere indecomposable module $X$ in $\mathcal{C}_0$, we
have $\Hom_{\widehat{B}}(X,P_{1,i})\neq 0$, since the socle of $P_{1,i}$ is a composition factor of $X$. Moreover, we have a canonical non-zero
homomorphism $f: P_{1,i}\to P_{2,i}=\nu_{\widehat{B}}(P_{1,i})$, since the top of $P_{1,i}$ is the socle of $P_{2,i}$. Observe that $P_{2,i}$ lies in
$\mathcal{R}_{3}=\nu_{\widehat{B}}(\mathcal{R}_1)$. Finally, since the component $\mathcal{C}_3$ separates
$\bigvee_{p<3}(\mathcal{C}_p\vee\mathcal{R}_p)$ from $\mathcal{R}_3\vee(\bigvee_{p>3}(\mathcal{C}_p\vee\mathcal{R}_p))$ in $\mod\widehat{B}$, we
conclude that the homomorphism $f$ factors through a module $M$ from the additive category $\add(\mathcal{C}_3)$ of $\mathcal{C}_3$. In particular, we
obtain a short path $X\to P\to Z$ with $P=P_{1,i}$ in $\mathcal{R}_1$ and $Z$ an indecomposable module in $\mathcal{C}_3$.
\end{itemize}
Consider now the push-down functor \cite{Ga} $F_{\lambda}: \mod\widehat{B}\to\mod A$ associated to the Galois covering
$F: \widehat{B}\to \widehat{B}/(\psi\nu_{\widehat{B}})$. Since, by the properties (a) - (d), the repetitive category $\widehat{B}$ is locally
support-finite, we conclude by the density theorem of \cite{DS} that the push-down functor $F_{\lambda}$ is dense. In particular, by \cite[Theorem 3.6]{Ga},
the Auslander-Reiten quiver $\Gamma_A$ of $A$ is the orbit quiver $\Gamma_{\widehat{B}}/(\psi\nu_{\widehat{B}})$ of $\Gamma_{\widehat{B}}$ with respect
to the induced action of $(\psi\nu_{\widehat{B}})$ on $\Gamma_{\widehat{B}}$. In fact, for $g=\psi\nu_{\widehat{B}}$, $G=(g)$, and the induced action
of $G$ on $\mod\widehat{B}$ (see \cite{Ga}), the push-down functor $F_{\lambda}: \mod\widehat{B}\to\mod A$ induces isomorphisms of $K$-modules
\[\bigoplus_{r\in\mathbb{Z}}\Hom_{\widehat{B}}(U,g^rV)\buildrel{\thicksim}\over{\hbox to 6mm{\rightarrowfill}}\Hom_{\widehat{B}}(F_{\lambda}(U),F_{\lambda}(V))\]
for all indecomposable modules $U$ and $V$ in $\mod\widehat{B}$. Since $g=\psi\nu_{\widehat{B}}$
with $\psi$ a positive automorphism of $\widehat{B}$ and $\nu_{\widehat{B}}(\mathcal{C}_0)=\mathcal{C}_2$,
we conclude that $\mathcal{C}_T=\mathcal{C}_0=F_{\lambda}(\mathcal{C}_0)$ is a regular component of $\Gamma_A$ consisting entirely of indecomposable
$B$-modules. Then $\mathcal{C}_T$ is a generalized standard component of $\Gamma_A$, because $\mathcal{C}_T$ is a generalized standard component of
$\Gamma_B$, and $B=A/\ann_A(\mathcal{C})$. This shows the required property (ii).

For (iii), assume first that $\psi=\varphi\nu_{\widehat{B}}$ for a positive automorphism $\varphi$ of $\widehat{B}$. Then $g=\varphi\nu_{\widehat{B}}^2$
and $g(\mathcal{C}_T)= g(\mathcal{C}_0)=\mathcal{C}_p$ for some $p\geq 4$. Hence, applying (a) and (g), we conclude that there is no short path in
$\mod\widehat{B}$ of the form $X\to Y\to Z$ with $X$ in $\mathcal{C}_0$, $Z$ in $g^r\mathcal{C}_0$ and $Y$ not in $g^r\mathcal{C}_0$, for any
$r\in\mathbb{Z}$. This shows that $\mathcal{C}_T=F_{\lambda}(\mathcal{C}_0)$ has no external short path in $\mod B$. On the other hand, if
$g=\psi\nu_{\widehat{B}}$ is not of the form $\varphi\nu_{\widehat{B}}^2$ with $\varphi$ a positive automorphism of $\widehat{B}$, then
$g(\mathcal{C}_0)=\mathcal{C}_2$ or $g(\mathcal{C}_0)=\mathcal{C}_3$. In the first case, $\mathcal{C}_T=F_{\lambda}(\mathcal{C}_0)$ is a sincere regular
component in $\mod A$, and then we have an external short path $X\to Q\to X$ of $\mathcal{C}_T$ for an arbitrary sincere indecomposable module $X$
in $\mathcal{C}_T$ and $Q$ an arbitrary indecomposable projective-injective $A$-module. For $g=\psi\nu_{\widehat{B}}$ with $g(\mathcal{C}_0)=\mathcal{C}_3$,
we conclude from (h) that there is an external short path $F_{\lambda}(X)\to F_{\lambda}(P)\to F_{\lambda}(Z)$ of
$\mathcal{C}_T=F_{\lambda}(\mathcal{C}_0)=F_{\lambda}(\mathcal{C}_3)$ for some modules $X\in\mathcal{C}_0$, $Z\in\mathcal{C}_3$ and an indecomposable
projective-injective $A$-module $F_{\lambda}(P)\in F_{\lambda}(\mathcal{R}_1)$. Therefore, the equivalence (iii) holds.
\end{proof}
We complete now the proof of Theorem \ref{thm10} by showing that the implication (i)$\Rightarrow$(ii) also holds.

Let $A$ be a self-injective algebra such that $\Gamma_{A}$ admits
a regular acyclic component $\mathcal{C}$ without external short
paths in $\mod A$. Then it follows from Corollary \ref{cor4} that
$\mathcal{C}$ is a generalized standard component of $\Gamma_A$.
Let $I$ be the annihilator $\ann_A({\mathcal{C}})$ of
$\mathcal{C}$ in $A$ and $B=A/I$. We may choose a set $e_1,...,
e_n$ of pairwise orthogonal primitive idempotents of $A$ such that
$1_A=e_1+...+e_n$ and $\{e_i| 1\leqslant i \leqslant m\}$, for some
$m\leqslant n$, is the set of all idempotents in $\{e_i| 1
\leqslant i \leqslant n\}$ which are not contained in $I$. Then
$e=e_1+...+e_m$ is an idempotent of $A$ such that $e+I$ is the
identity $1_B$ of $B$, uniquely determined by $I$ up to an inner
automorphism of $A$, called the {\em residual identity} of
$B=A/I$. Moreover, $B \cong eAe/eIe$. Further, it follows from
Corollary \ref{cor3} that $B$ is  a tilted algebra of the form $\End_H(T)$, for
a (wild) hereditary algebra $H$ and a regular tilting $H$-module $T$,
and $\mathcal{C}$ is the connecting component $\mathcal{C}_T$ of
$\Gamma_B$ determined by $T$. In particular, the valued quiver
$Q_B$ of $B$ is acyclic. Applying \cite[Proposition 2.3 and
Theorem 5.1]{SY1}, we obtain also that $Ie=l_A(I)$ and
$eI=r_A(I)$, where $l_A(I)$ and $r_A(I)$ denote the left
annihilator and the right annihilator of $I$ in $A$, respectively,
and $\soc(A)$  is contained in $I$. We note that by a theorem of
Nakayama \cite{Na} the left socle $\soc(_AA)$ and right socle
$\soc(A_A)$ of $A$ coincide, and are denoted by $\soc(A)$, which
is an ideal of $A$. Observe that $IeI=0$, and so $Ie$ is a right
$B$-module and $eI$ is a left $B$-module. In fact, by
\cite[Proposition 2.3]{SY1}, $Ie$ is an injective cogenerator in
$\mod B$ and $eI$ is an injective cogenerator in $\mod B^{\op}=
B-\mod$ (the category of finitely generated left $B$-modules). We
have also $l_{eAe}(I)=eIe=r_{eAe}(I)$, and so $I$ is a deforming
ideal in the sense of \cite[(2.1)]{SY1}. The canonical isomorphism
of algebras $eAe/eIe \rightarrow A/I$ allows to consider $I$ as an
$(eAe/eIe)-(eAe/eIe)$-bimodule. Denote by $A[I]$ the direct sum
$(eAe/eIe) \oplus I$ of $K$-modules with the multiplication
\[(b,x)\cdot(c,y)=(bc,by+xc+xy)\]
for $b,c \in eAe/eIe$ and $x,y \in I$. Then, by \cite[Theorem 4.1]{SY1}, $A[I]$ is a basic connected
self-injective artin algebra with the identity $(e+eIe, 1_A-e)$, and the same Nakayama permutation as $A$.
We note that $\{(e_i+eIe,0)| 1\leqslant i\leqslant m\} \cup \{(0,e_j)| m< j \leqslant n\}$ is a complete
set of pairwise orthogonal idempotents of $A[I]$ whose sum is the identity $1_{A[I]}$ of $A[I]$. Further,
by identifying  $x \in I$ with $(0,x) \in A[I]$, we may consider $I$ as an ideal of $A[I]$. Then
$e=(e+eIe,0)$ is a residual identity of $A[I]/I= eAe/eIe \buildrel{\thicksim}\over{\hbox to 6mm{\rightarrowfill}}
A/I$, $eA[I]e=(eAe/eIe)\oplus eIe$ and the canonical algebra epimorphism $eA[I]e \rightarrow eA[I]e/eIe$
is a retraction. Finally, $IeI=(0, IeI)=0$ and $Ie=l_{A[I]}(I)$, $eI=r_{A[I]}(I)$. Applying now
\cite[Theorem 3.8]{SY2} we conclude that $A[I]$ is isomorphic to the orbit algebra $\widehat{B}/(\psi\nu_{\widehat{B}})$
for a positive automorphism $\psi$ of $\widehat{B}$.

It follows also from \cite[Theorem 4.1]{SY1} that  there is a canonical algebra isomorphism
$\phi : A[I]/ \soc(A[I]) \rightarrow A/\soc(A)$. Observe that $\soc(A[I])=\soc (I)= \soc (A)$,
since $\soc(A)$ and $\soc(A[I])$ are contained in $I$. In particular, we obtain the isomorphism
$\phi^{\ast}: \mod (A[I]/ \soc(A[I])) \rightarrow \mod(A/ \soc(A))$ of module categories, induced by $\phi$.
Note that the regular component $\mathcal{C}=\mathcal{C}_T$ of $\Gamma_A$ consists entirely of
$B$-modules, and hence is a component of the Auslander-Reiten quiver $\Gamma_{A/\soc(A)}$ of
$A/\soc (A)$. Clearly, $\mathcal{C}$ is a regular acyclic component of $\Gamma_{A/\soc (A)}$
without external short paths in $\mod (A/\soc(A))$. Denote by $\mathcal{C}'$ the regular  acyclic component
of the Auslander-Reiten quiver $\Gamma_{A[I]/ \soc(A[I])}$ of $A[I]/ \soc(A[I])$ such that
$\phi^{\ast}(\mathcal{C}')=\mathcal{C}$. In fact, we have $I= \ann_{A[I]}(\mathcal{C}')$ and
$A[I]/I \cong eAe/eIe \buildrel{\thicksim}\over{\hbox to 6mm{\rightarrowfill}} A/I=B$ canonically,
so $\mathcal{C}'$ is the regular connecting component of the tilted algebra $eAe/eIe$, isomorphic to $B$.
We claim that $\mathcal{C}'$ has no external short path in $\mod A[I]$. Suppose there is a short path
$X \rightarrow Y \rightarrow Z$ in $\mod A[I]$ with $X$ and $Z$ in $\mathcal{C}'$ but $Y$ not in
$\mathcal{C}'$. Obviously, if $Y$ is not a projective module in $A[I]$, then $Y$ is a module in
$\mod (A[I]/\soc(A[I]))$, and we obtain an external short path $\phi^{\ast}(X) \rightarrow \phi^{\ast}(Y)
\rightarrow \phi^{\ast}(Z)$ of $\mathcal{C}$ in $\mod A$, a contradiction. Assume $Y=P$ is an indecomposable
projective, hence injective, module in $\mod A[I]$. Then there is an indecomposable projective-injective
module $Q$ in $\mod A$ such that $\phi^{\ast}(P/\soc(P))=Q/\soc(Q)$. Since the Nakayama permutations of
the algebras $A[I]$ and $A$ coincide, we conclude also that $\phi^{\ast}(\soc(P))= \soc(Q)$ and
$\phi^{\ast}( \top (P))= \top (Q)$.

Further, since there are non-zero homomorphisms $X\to P$ and $P\to Y$ in $\mod A[I]$, we conclude that $\soc(P)$ is a simple composition factor of $X$
and $\top(P)$ is a simple composition factor of $Y$. But then $\soc(Q)$ is a simple composition factor of $\phi^*(X)$ and $\top(Q)$ is a simple composition
factor of $\phi^*(Z)$. Hence we obtain a short path $\phi^*(X)\to Q\to \phi^*(Z)$ in $\mod A$, with $\phi^*(X)$ and $\phi^*(Z)$ in $\mathcal{C}=\mathcal{C}_T$ and
$Q$ not in $\mathcal{C}$, because $Q$ is projective and $\mathcal{C}$ is regular. Thus $\mathcal{C}$ admits an external short path in $\mod A$.
Summing up, $\mathcal{C}'$ is a regular acyclic component of $\Gamma_{A[I]/\soc(A[I])}$ without external short paths in $\mod (A[I]/\soc(A[I]))$.

Therefore, applying Proposition \ref{prop4.1} to $\widehat{B}/(\psi\nu_{\widehat{B}})$, isomorphic to $A[I]$, we conclude that $A[I]$ is isomorphic to
$\widehat{B}/(\varphi\nu_{\widehat{B}}^2)$ for a positive automorphism $\varphi$ of $\widehat{B}$. Invoking now the structure of the Auslander-Reiten
quiver of $\widehat{B}/(\psi\nu_{\widehat{B}})=\widehat{B}/(\varphi\nu_{\widehat{B}}^2)$ described in the proof of Proposition \ref{prop4.1}, we
conclude that, for any indecomposable projective-injective module $P$ in $\mod A[I]$, the top $\top(P)$ of $P$ is not isomorphic to the socle
$\soc(P)$ of $P$. Since the Nakayama permutations of $A$ and $A[I]$ coincide, we conclude that for any indecomposable projective-injective module $Q$
in $\mod A$, we have that the $\top(Q)$ of $Q$ is not isomorphic to the socle $\soc(Q)$ of $Q$. In particular, if $\nu$ is the Nakayama permutation of $A$
(so of the set $\{1,\ldots, n\}$), and $e_i$ is a primitive summand of $e$ (so $i\in\{1,\ldots, m\}$) then $e_i\neq e_{\nu(i)}$. Applying
\cite[Proposition 3.2]{SY3} we then conclude that the algebras $A$ and $A[I]$ are isomorphic. Therefore, $A$ is isomorphic to
$\widehat{B}/(\varphi\nu_{\widehat{B}}^2)$. This proves that the implication (i)$\Rightarrow$(ii) in Theorem \ref{thm10}  holds.

We end this section with comments concerning the structure of the Auslander-Reiten quivers of self-injective algebras having regular acyclic
components without external short paths.

Let $H$ be a wild hereditary algebra, $T$ a multiplicity-free regular tilting $H$-module, $B=\End_H(T)$ the associated tilted algebra, $\varphi$ a positive
automorphism of $\widehat{B}$, and $A=\widehat{B}/(\varphi\nu_{\widehat{B}}^2)$. Moreover, let $Q=Q_H$ be the valued quiver of $H$ and $\Delta=Q^{\rm op}$.
We use the notation introduced in the proof of Proposition \ref{prop4.1} for description the structure of the Auslander-Reiten quiver $\Gamma_{\widehat{B}}$
of the repetitive category $\widehat{B}$ of $B$. Since the push-down functor $F_{\lambda}: \mod\widehat{B}\to\mod A$, associated to the Galois covering
$F: \widehat{B}\to \widehat{B}/(\varphi\nu_{\widehat{B}}^2)$, is dense and preserves the Auslander-Reiten sequences, we conclude that
$\Gamma_A=\Gamma_{\widehat{B}/(\varphi\nu_{\widehat{B}}^2)}= \Gamma_{\widehat{B}}/(\varphi\nu_{\widehat{B}}^2)$ and may be visualised as follows 
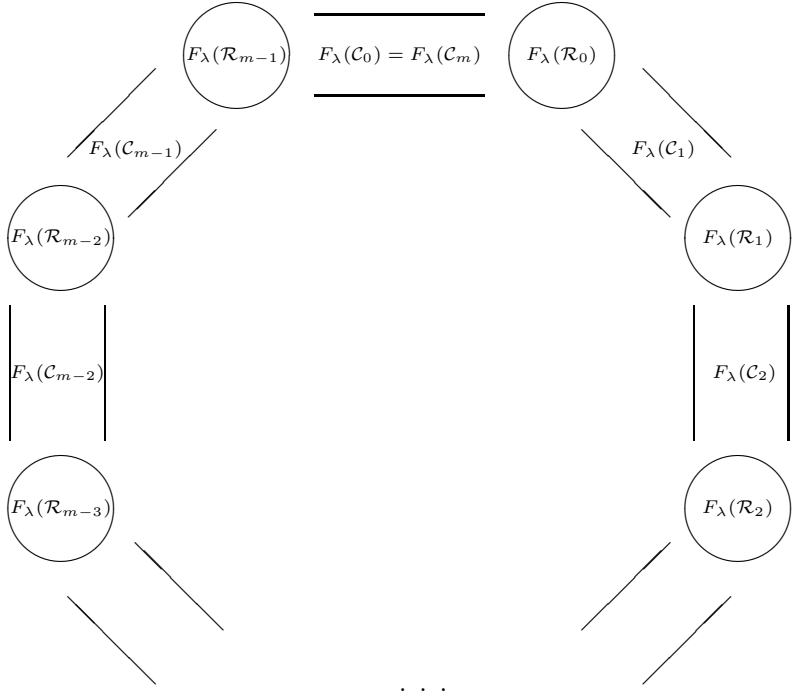
\begin{center}
\setlength{\unitlength}{0.9mm}
\begin{picture}(100,120)(-40,-50)

\put(-12.5,49){\line(1,0){25}}
\put(-12.5,61){\line(1,0){25}}
\put(0,55){\makebox(0,0){\scriptsize{$F_{\lambda}(\mathcal{C}_0)= F_{\lambda}(\mathcal{C}_m)$}}}
\put(24,55){\circle{30}}
\put(-24,55){\circle{20}}
\put(24,55){\makebox(0,0){\scriptsize{$F_{\lambda}(\mathcal{R}_{0})$}}}
\put(-24,55){\makebox(0,0){\scriptsize{$F_{\lambda}(\mathcal{R}_{m-1})$}}}

\put(27,44){\line(1,-1){13}}
\put(36,53){\line(1,-1){13}}
\put(39,41){\makebox(0,0){\scriptsize{$F_{\lambda}(\mathcal{C}_1)$}}}
\put(-27,44){\line(-1,-1){13}}
\put(-36,53){\line(-1,-1){13}}
\put(-39,41){\makebox(0,0){\scriptsize{$F_{\lambda}(\mathcal{C}_{m-1})$}}}
\put(50,28){\circle{30}}
\put(-50,28){\circle{30}}
\put(-50,28){\makebox(0,0){\scriptsize{$F_{\lambda}(\mathcal{R}_{m-2})$}}}
\put(50,28){\makebox(0,0){\scriptsize{$F_{\lambda}(\mathcal{R}_{1})$}}}

\put(43.5,-2){\line(0,1){20}}
\put(57.5,-2){\line(0,1){20}}
\put(51,8){\makebox(0,0){\scriptsize{$F_{\lambda}(\mathcal{C}_{2})$}}}

\put(-43.5,-2){\line(0,1){20}}
\put(-57.5,-2){\line(0,1){20}}
\put(-50.5,8){\makebox(0,0){\scriptsize{$F_{\lambda}(\mathcal{C}_{m-2})$}}}

\put(50,-12){\circle{30}}
\put(-50,-12){\circle{30}}
\put(50,-12){\makebox(0,0){\scriptsize{$F_{\lambda}(\mathcal{R}_{2})$}}}
\put(-50, -12){\makebox(0,0){\scriptsize{$F_{\lambda}(\mathcal{R}_{m-3})$}}}

\put(27,-30){\line(1,1){13}}
\put(36,-38){\line(1,1){13}}

\put(-26,-30){\line(-1,1){13}}
\put(-36,-38){\line(-1,1){13}}

\multiput(0,-39)(3,0){3}{.}

\end{picture}
\end{center}

\noindent for a positive integer $m\geq 4$, where, for each
$i\in\{0, 1, \ldots, m-1\}$, $F_{\lambda}(\mathcal{C}_i)$ is an
acyclic component with the stable part of the form
$\mathbb{Z}\Delta$ and $F_{\lambda}(\mathcal{R}_i)$ is an infinite
family of components with the stable parts of the form
$\mathbb{Z}\mathbb{A}_{\infty}$, containing at least one
indecomposable projective module and a simple module. Concerning
the distribution of simple and projective modules in the
components $F_{\lambda}(\mathcal{C}_i)$, $i\in\{0, 1, \ldots,
m-1\}$, we have two cases (see \cite[Section 5]{EKS}).

{\bf Case 1.} The connecting component $\mathcal{C}_T$ of $\Gamma_B$ contains a simple $B$-module. Then $m=2r$, for $r\geq 2$,
$\varphi\nu_{\widehat{B}}^2=\varrho\nu_{\widehat{B}}^r$ for a rigid automorphism $\varrho$ of $\widehat{B}$, the components $F_{\lambda}(\mathcal{C}_{2i})$,
$i\in\{0, 1, \ldots, r-1\}$, are regular and contain a simple $A$-module, while the components $F_{\lambda}(\mathcal{C}_{2i+1})$,
$i\in\{0, 1, \ldots, r-1\}$, contain a projective module but not a simple module.

We also mention there is always a regular tilting $H$-module $T^*$ such that the connecting component $\mathcal{C}_{T^*}$ of the tilted algebra
$B^*=\End_H(T^*)$ determined by $T^*$ contains a simple $B^*$-module (see \cite[Proposition 2.5]{EKS}).
For a concrete example of such tilted algebra we refer to \cite[Example 6.2]{KS}.

{\bf Case 2.} The connecting component $\mathcal{C}_T$ of $\Gamma_B$ does not contain a simple $B$-module. Then all components
$F_{\lambda}(\mathcal{C}_i)$, $i\in\{0, 1, \ldots, m-1\}$, are regular and without simple modules, and consequently all simple and projective modules are
located in the parts $F_{\lambda}(\mathcal{R}_i)$, $i\in\{0, 1, \ldots, m-1\}$. We note that it is an open problem if in this case $m$ is even,
or equivalently, $\varphi\nu_{\widehat{B}}=\varrho\nu_{\widehat{B}}^r$ for some positive integer $r\geq 2$ and a rigid automorphism $\varrho$
of $\widehat{B}$.

We also mention that, by \cite[Corollary 4]{KS}, there are infinitely many pairwise non-isomorphic tilted algebras $B^*=\End_H(T^*)$ given by regular
tilting $H$-modules $T^*$ such that the connecting component $\mathcal{C}_{T^*}$ of $\Gamma_{B^*}$ determined by $T^*$ does not contain simple modules.

Finally, we note that all regular components of the form $\mathbb{Z}\Delta$ in $\Gamma_A$ have no external short paths. Similarly, for $m\geq 5$,
all non-regular components with the stable parts $\mathbb{Z}\Delta$ in $\Gamma_A$ have no external short paths. On the other hand, for $m=4$ and
$\mathcal{C}_T$ containing a simple module, the non-regular components $F_{\lambda}(\mathcal{C}_1)$ and $F_{\lambda}(\mathcal{C}_3)$ have
external short paths in $\mod A$.

\end{document}